\documentclass[
final,
nomarks
]{dmtcs-episciences}

\usepackage[utf8]{inputenc}
\usepackage{subfigure}

\usepackage{amsmath} 
\usepackage{amsthm} 
\usepackage{amsfonts} 
\usepackage{amssymb} 
\usepackage{mathtools} 
\usepackage{bm} 
\usepackage{bbm} 

\usepackage{paralist} 
\usepackage{array} 

\def\expect{\ensuremath{\mathbb{E}}}
\def\cL{\ensuremath{\mathfrak{L}}}
\def\cN{\ensuremath{\mathfrak{N}}}
\def\cQ{\ensuremath{\mathfrak{Q}}}
\def\cU{\ensuremath{\mathfrak{U}}}
\def\fS{\ensuremath{\mathfrak{S}}}
\def\One{\ensuremath{\mathbbm{1}}}

\DeclareMathOperator{\Id}{Id}
\DeclareMathOperator{\Rev}{Rev}
\DeclareMathOperator{\Ext}{Ext}

\newtheoremstyle{plain}
  {\medskipamount}
  {\smallskipamount}
  {\sl}
  {0pt}
  {\bfseries}
  {.}
  { }
  {\thmname{#1}\thmnumber{ #2}{\normalfont\thmnote{ (#3)}}}

\theoremstyle{plain}

\newtheorem{theorem}{Theorem}[section]
\newtheorem{lemma}[theorem]{Lemma}
\newtheorem{proposition}[theorem]{Proposition}
\newtheorem{corollary}[theorem]{Corollary}
\newtheorem{conjecture}[theorem]{Conjecture}
\newtheorem{problem}[theorem]{Problem}
\newtheorem{question}[theorem]{Question}
\newtheorem{definition}[theorem]{Definition}
\newtheorem{notation}[theorem]{Notation}

\newtheorem*{remark*}{Remark}

\newenvironment{proofsketch}[1][]{%
\begingroup
\begin{proof}[#1]}%
{\end{proof}\endgroup}

\newenvironment{proofcases}[1][1]{%
\begingroup\def\case{\item}
\begin{asparaenum}[\emph{Case} #1.]}%
{\end{asparaenum}\medskip\endgroup}

\usepackage[round]{natbib}

\author{Josefran de Oliveira Bastos\thanks{Supported by Coordena\c{c}\~{a}o de
  Aperfei\c{c}oamento de Pessoal de N\'{\i}vel Superior (CAPES/Brazil).}
  \and Leonardo Nagami Coregliano\thanks{Supported by S\~{a}o Paulo Research Foundation
  (FAPESP), grant~\#2013/23720-9.%
}}
\title[Packing densities and minimum number of monotone sequences]{%
  Packing densities of layered permutations and the minimum number of monotone sequences in layered permutations%
}
\affiliation{
  Instituto de Matem\'{a}tica e Estat\'{\i}stica, Universidade de S\~{a}o Paulo, Brazil
}
\received{2015-10-28}
\revised{2016-06-14}
\accepted{2016-06-17}
\begin{document}
\publicationdetails{18}{2016}{2}{7}{1313}
\maketitle
\begin{abstract}
    In this paper, we present two new results of layered permutation densities. The first one generalizes
    theorems from H\"{a}st\"{o} (2003) and Warren (2004) to compute the permutation packing of permutations
    whose layer sequence is~$(1^a,\ell_1,\ell_2,\ldots,\ell_k)$ with~$2^a-a-1\geq k$ (and similar
    permutations). As a second result, we prove that the minimum density of monotone sequences of
    length~$k+1$ in an arbitrarily large layered permutation is asymptotically~$1/k^k$. This value is
    compatible with a conjecture from Myers (2003) for the problem without the layered restriction (the same
    problem where the monotone sequences have different lengths is also studied).
\end{abstract}

\section{Introduction}

As usual, a \emph{permutation} over~$[n] = \{1,2,\ldots,n\}$ is a bijective function
of~$[n]$ onto itself. We denote the set of all permutations over~$[n]$ by~$\fS_n$
and, for every~$\sigma\in\fS_n$, we say that the \emph{length} of~$\sigma$
(denoted~$\lvert\sigma\rvert$) is~$n$.

We denote by~$\naturals$ the set of non-negative integers and
let~$\naturals^*=\naturals\setminus\{0\}$. We also denote by~$\fS = \bigcup_{n\in\naturals}\fS_n$ the
set of all finite permutations.

We also use the notation~$\{i_1,i_2,\ldots,i_k\}_<\subset A$ as a shorthand
for~$\{i_1,i_2,\ldots,i_k\}\subset A$ with~$i_1<i_2<\cdots<i_k$. Furthermore, we
frequently denote a permutation~$\sigma\in\fS_n$
by~$(\sigma(1)\sigma(2)\cdots\sigma(n))$ using extra parentheses whenever the
notation starts to get too ambiguous.

Let~$\sigma\in\fS_n$ and let~$\{i_1,i_2,\ldots,i_m\}_<\subset[n]$, the
\emph{subpermutation induced} by~$\{i_1,i_2,\ldots,i_m\}$ in~$\sigma$ is the unique
permutation~$\tau\in\fS_m$ such that for every~$j,k\in[m]$ we
have~$\sigma(i_j)<\sigma(i_k)$ if and only if~$\tau(j) < \tau(k)$ and it is
denoted by~$\sigma[\{i_1,i_2,\ldots,i_m\}]$. For example, if~$\sigma = (68153427)$,
then we have~$\sigma[\{1,3,6\}] = (312)$ and~$\sigma[\{2,4,7,8\}] = (4213)$.

Furthermore, if~$\tau\in\fS_m$ and~$\sigma\in\fS_n$, then we define the
\emph{number of occurrences}~$\Lambda(\tau, \sigma)$ of the permutation~$\tau$
in~$\sigma$ as
\begin{align*}
  \Lambda(\tau,\sigma)
  & =
  \left\lvert \left\{A \subset [n] : \sigma[A] = \tau\right\}\right\rvert
  =
  \left\lvert \left\{A\in\binom{[n]}{m} : \sigma[A] = \tau\right\}\right\rvert,
\end{align*}
where
\begin{align*}
  \binom{[n]}{m} & = \{A\subset [n] : \left\lvert A\right\rvert = m\}.
\end{align*}

We also define the \emph{density}~$p(\tau, \sigma)$ of~$\tau$ in~$\sigma$ as
\begin{align*}
  p(\tau,\sigma) & =
  \begin{dcases*}
    \binom{n}{m}^{-1}\Lambda(\tau,\sigma), & if~$m\leq n$;\\
    0, & if~$m > n$;
  \end{dcases*}
\end{align*}
which, when~$m\leq n$, coincides with the probability that we induce~$\tau$ by
picking uniformly at random an element of~$\binom{[n]}{m}=\{A\subset[n]:|A|=m\}$. We
extend the definitions of~$\Lambda$ and~$p$ linearly in the first coordinate to
(finite) formal linear combinations of elements of~$\fS$ (\textit{i.e.}, we extend their
domain to~$\reals\fS\times\fS$). Let~$f\in\reals\fS$ be a formal linear combination of
permutations. For every~$N\in\naturals$, define
\begin{align*}
  p_N(f) & = \max\{p(f,\sigma) : \sigma\in\fS_N\};\\
  \Ext_N(f) & = \{\sigma\in\fS_N : p(f,\sigma) = p_N(f)\};\\
  p(f) & = \lim_{N\to\infty}p_N(f).
\end{align*}

The following argument, which has already become part of the folklore of extremal
combinatorics, proves that the above limit indeed exists.

\begin{lemma}
  For every~$f\in\reals\fS$, there exists~$N_0\in\naturals$ such that~$p_N(f)\geq p_{N+1}(f)$
  for every~$N\geq N_0$.
\end{lemma}

\begin{proof}
  Write~$f = \sum_{i=1}^kc_i\tau_i$ and let~$N_0 = \max\{\lvert\tau_i\rvert :
  i\in[k]\}$.

  Note that if~$N\geq N_0$ and~$\sigma\in\fS_{N+1}$ is such
  that~$p_{N+1}(f)=p(f,\sigma)$, then we can compute~$p(\tau_i,\sigma)$ using the
  following random experiment. We first pick~$\bm{j}\in[N+1]$ uniformly at random
  then we compute the density of~$\tau_i$ in~$\sigma[[N+1]\setminus\{\bm{j}\}]$. Note
  that
  \begin{align*}
    p(\tau_i,\sigma) & = \expect[p(\tau_i,\sigma[[N+1]\setminus\{\bm{j}\}])].
  \end{align*}

  But then, by linearity of expectation, we have
  \begin{align*}
    p_{N+1}(f) & = p(f,\sigma) = \expect[p(f,\sigma[[N+1]\setminus\{\bm{j}\}])]
    \leq p_N(f),
  \end{align*}
  as desired.
\end{proof}

One of the most studied problems involving permutations is the \textit{packing density
  problem} stated below.

\begin{problem}[Permutation packing]
  Let~$f\in\reals\fS$ be a formal linear combination of permutations. For
  every~$N\in\naturals$, how large can~$p(f,\sigma)$ be and what are the properties that
  a~$\sigma\in\Ext_N(f)$ has?
\end{problem}

In this problem, one interesting subfamily of~\fS\ is the family of \emph{layered
  permutations}. A permutation~$\sigma\in\fS_n$ of length~$n$ is called
\emph{layered} if there exists~$\{i_1,i_2,\ldots,i_k\}_<\subset[n+1]$ with~$i_1=1$
and~$i_k=n+1$ such that
\begin{align*}
  \sigma[\{i_j,i_j+1,\ldots,i_{j+1}-1\}] = ((i_{j+1}-1)(i_{j+1}-2)\cdots(i_j + 1)i_j),
\end{align*}
for every~$j\in[k-1]$, and~$\sigma(a) < \sigma(b)$ whenever~$a,b\in[n]$ are such
that~$a < i_j \leq b$ for some~$j\in[k-1]$.

This means that the permutation~$\sigma$ consists of an increasing sequence of
decreasing sequences. Such decreasing sequences are called~\emph{layers}
of~$\sigma$. We denote a layered permutation by the sequence of lengths of its layers
as~$(\ell_1,\ell_2,\ldots,\ell_k)$ (it is easy to see that such representation is
unique) and we call such a sequence the \emph{layer sequence} or \emph{layer
  decomposition} of~$\sigma$.

Note that the function that maps a layered permutation to its layer sequence provides
an isomorphism to the theory of compositions (\textit{i.e.}, ordered partitions of~$n$).

An \emph{antilayer} of a layered permutation~$\sigma$ is a maximal contiguous
non-empty sequence of layers of length~$1$. A \emph{block} of a layered permutation
is either a layer of length at least~$2$ or an antilayer, \textit{i.e.}, a block is a maximal
monotone interval. When it is convenient, we
also denote a layered permutation by the sequence of lengths of its blocks using a
``hat'' to denote when the corresponding block is an antilayer, and we call this
sequence, the \emph{block sequence} or the \emph{block decomposition} of the
permutation, \textit{e.g.}, the permutation~$\sigma=(321457689)$ has block
decomposition~$(3,\widehat{2},2,\widehat{2})$. The theorem below shows that the
problem of packing a layered permutation is much easier than the general case.

Throughout this paper, we let~$\reals\fS$ denote the set of formal linear combinations
of elements of~$\fS$ with real coefficients and we say that~$f \in \reals \fS$ is a
conical combination when all its coefficients are non-negative.

\begin{theorem}[\cite{AlbertAtkinsonHandleyHoltonStromquist:PackingPermutations}]
  \label{thm:extlayer}
  If~$f\in\reals\fS$ is a conical combination of layered permutations then, for
  every~$N\in\naturals$, there exists a layered permutation in~$\Ext_N(f)$.  This in
  particular means that
  \begin{align*}
    p_N(f) & =
    \max\{p(f,\sigma) : \sigma\in\fS_N \text{ is layered}\}.
  \end{align*}
\end{theorem}

Based on the theorem above, \cite{Price:PackingDensitiesOfLayeredPatterns}
suggested an algorithm that, given a conical combination~$f$ of layered permutations,
computes lower bounds to~$p(f)$ converging to~$p(f)$.

Much more work has been done to compute the packing density problem of layered permutations than of
non-layered permutations. \cite{AlbertAtkinsonHandleyHoltonStromquist:PackingPermutations} solved the problem
of layered permutations with two layers. \cite{Hasto:PackingOtherLayered} solved the problem for any layered
permutation that has~$r$ layers of length at least~$\log{r+1}$. He also solved the problem for all layered
permutations of the form~$(k,1,k)$ with~$k\geq 3$. Our first result in this work is the following.

\begin{theorem}\label{thm:smallanti}
  Let~$a,\ell_1,\ell_2,\ldots,\ell_k\in\naturals^*$ be positive integers such that~$2\leq
  a\leq \ell_1\leq\ell_2\leq\cdots\leq\ell_k$ and~$2^a-a-1\geq k$.
  If~$\sigma$ is the layered permutation~$(\widehat{a},\ell_1,\ell_2,\ldots,\ell_k)$,
  then we have
  \begin{align*}
    p(\sigma) & =  \frac{|\sigma|!}{|\sigma|^{|\sigma|}}\,
    \frac{a^a}{a!}\prod_{i=1}^k\frac{\ell_i^{\ell_i}}{\ell_i!}.
  \end{align*}
\end{theorem}

The packing density problem is focused on estimating the permutation that maximizes
the number of occurrences of a smaller permutation. It is natural to think of the dual
problem, \textit{i.e.}, trying to find the permutation that \emph{minimizes} the number of
occurrences of a smaller permutation.

\begin{problem}[Permutation minimization]
  Let~$f\in\reals\fS$ be a formal linear combination of permutations. For
  every~$N\in\naturals$, define
  \begin{align*}
    p_N'(f) & = \min\{p(f,\sigma) : \sigma\in\fS_N\};\\
    \Ext_N'(f) & = \{\sigma\in\fS_N : p(f,\sigma) = p'_N(f)\};\\
    p'(f) & = \lim_{N\to\infty}p_N'(f).
  \end{align*}
  The problem then consists of computing explicitly the values above, which basically
  means answering the question: how small can~$p(f,\sigma)$ be and what are the
  properties that a~$\sigma \in \Ext_N'(f)$ has?
\end{problem}

Although we can restate the minimization problem as a permutation packing problem, we
lose the result from Theorem~\ref{thm:extlayer} by doing so. So the next problem is a
completely different problem.

\begin{problem}[Layered permutation minimization]
  Let~$f\in\reals\fS$ be a formal linear combination of layered permutations. For
  every~$N\in\naturals$, define
  \begin{align*}
    p_N''(f) & = \min\{p(f, \sigma) : \sigma \in \fS_N \text{ and }\sigma\text{ is layered}\};\\
    \Ext_N''(f) & = \{\sigma\in\fS_N \text{ layered}: p(f,\sigma) = p''_N(f)\};\\
    p''(f) & = \lim_{N\to\infty}p_N''(f).
  \end{align*}

  The problem then consists of computing explicitly the values above, which basically
  means answering the question: how small can~$p(f,\sigma)$ be if~$\sigma$ is
  layered and what are the properties that a~$\sigma \in \Ext_N''(f)$ has?
\end{problem}

For every length~$n\in\naturals$, two particular permutations deserve special notation. One
is the \emph{identity}, denoted by~$\Id_n=(12\cdots n)$, and the other is the
\emph{reverse}, denoted by~$\Rev_n=(n(n-1)\cdots 1)$.

A well-known theorem by \cite{ErdosSzekeres:ProblemInGeometry} states that every permutation
of~$k^2+1$ elements must contain a monotone subsequence of length~$k+1$. Later, Myers
proved the following quantitative version of this theorem.

\begin{theorem}[\cite{Myers:MinimumMonotoneSubsequences}]
  We have
  \begin{align*}
    p'(\Id_3 + \Rev_3) & = \frac{1}{4}.
  \end{align*}
  Furthermore, for every~$k\geq 2$, we have
  \begin{align*}
    \lim_{N\to\infty}\min\{p(\Rev_{k+1},\sigma) : \sigma\in\fS_N
    \text{ with } p(\Id_{k+1},\sigma) = 0\} & = \frac{1}{k^k}.
  \end{align*}
\end{theorem}

The second part of the theorem above led Myers to conjecture
that~$p'(\Id_{k+1}+\Rev_{k+1})=1/k^k$ for every~$k\geq 2$ (the case~$k=2$ is the
first part of the theorem). We state a generalized version of Myers' Conjecture
below.

\begin{conjecture}%
  [Generalization of \cite{Myers:MinimumMonotoneSubsequences}]
  \label{conj:myers}
  For every~$k,\ell\geq 2$, we have
  \begin{align*}
    p'(\Id_{\ell+1}+\Rev_{k+1}) & =
    \min\left\{\frac{1}{k^\ell},\frac{1}{\ell^k}\right\}.
  \end{align*}
\end{conjecture}

The minimum above is actually an upper bound to~$p'(\Id_{\ell+1}+\Rev_{k+1})$, obtained
by considering permutations that have only one of the monotone sequences counted.

Using the tool of Flag Algebras developed by \cite{Razborov:FlagAlgebras}, the case~$k=\ell=3$ was proved to
be true by \cite{BaloghHuLidickyPikhurkoUdvariVolec:MinimumMonotone4}.

We also remark that \cite{SamotijSudakov:NumberMonotoneSequences}
proved a version of this conjecture when~$n\leq k^2 + ck^{3/2}/\log k$ and~$k$
and~$c$ are sufficiently large.

As our second result, we solve the simpler version of this problem when we restrict
ourselves to the class of layered permutations, that is, we
compute~$p''(\Id_{\ell+1}+\Rev_{k+1})$.

\begin{theorem}\label{thm:minmono}
  If~$k\geq\ell\geq 3$, then we have~$p''(\Id_\ell + \Rev_k) = 1/(\ell-1)^{k-1}$. In
  particular, for~$m\geq 2$, we have~$p''(\Id_{m+1}+\Rev_{m+1}) = 1/m^m$.
\end{theorem}

The paper is organized as follows. In Section~\ref{sec:genprice}, we present Price's
Algorithm, one generalization of it, and the first result, which concerns permutation
packing. In Section~\ref{sec:min}, we present the second result, which concerns
minimization of the asymptotic density of monotone sequences in layered
permutations. Finally, in Section~\ref{sec:conc}, we conclude the text by presenting
some related work and conjectures.

\section{Generalizing Price's Algorithm}
\label{sec:genprice}


We start by presenting the original Price's Algorithm to compute a lower bound for
the packing density of a layered permutation. The idea behind the algorithm is to
note that the density of a layered permutation in an arbitrarily large layered
permutation with a bounded number of layers can be expressed by a polynomial.

Let~$\tau\in\fS_m$ be a layered permutation with layer
sequence~$(\ell_j)_{j=1}^k$. We define the \emph{Price Polynomial} of order~$n\in\naturals^*$
for~$\tau$ to be the real polynomial over~$n$ variables given by
\begin{align*}
  q_{n,\tau}(x_1,x_2,\ldots,x_n) & =
  m!\sum_{\{i_1,i_2,\ldots,i_k\}_<\subset[n]}\prod_{j=1}^k\frac{x_{i_j}^{\ell_j}}{\ell_j!}
  = \binom{m}{\ell_1,\ell_2,\ldots,\ell_k}
  \sum_{\{i_1,i_2,\ldots,i_k\}_<\subset[n]}\prod_{j=1}^kx_{i_j}^{\ell_j},
\end{align*}
where an empty sum is taken to have result~$0$ and an empty product is taken to have result~$1$.

We also define the Price Polynomial of order~$n\in\naturals^*$ for a linear
combination of layered permutations~$f = \sum_{i=1}^ka_i\tau_i$ as
\begin{align*}
  q_{n,f} & \equiv \sum_{i=1}^ka_iq_{n,\tau_i}.
\end{align*}

The \emph{Price Bound} of order~$n\in\naturals^*$ for a linear combination of layered
permutations~$f$ is the value
\begin{align*}
  \cL_{n,f} & = \max\left\{q_{n,f}(x_1,x_2,\ldots,x_n) : \sum_{j=1}^nx_j = 1
  \text{ and } \forall j\in[n], x_j\geq 0\right\}.
\end{align*}

\begin{remark*}
  This maximum exists since the set~$\{(x_1,x_2,\ldots,x_n)\in[0,1]^n : \sum_{j=1}^n
  x_j = 1\}$ is compact and~$q_{n,f}$ is continuous.
\end{remark*}

\begin{proposition}[\cite{Price:PackingDensitiesOfLayeredPatterns}]
  \label{prop:Pricealgo}
  If~$f\in\reals\fS$ is a conical combination of layered permutations, then, for
  every~$n\in\naturals^*$, we have~$\cL_{n,f}\leq\cL_{n+1,f}\leq p(f)$.
\end{proposition}

The proposition above suggests a simple algorithm of computing approximations for the
values~$\cL_{n,f}$ (which are maxima of polynomials) to get lower bounds
for~$p(f)$. This is why this is called ``Price's Algorithm''.

\begin{corollary}[of Theorem~\ref{thm:extlayer} and Proposition~\ref{prop:Pricealgo},
    \cite{Price:PackingDensitiesOfLayeredPatterns}]
  \label{cor:priceconverges}
  If~$f\in\reals\fS$ is a conical combination of layered permutations, then we have
  \begin{align*}
    \lim_{n\to\infty}\cL_{n,f} & = p(f).
  \end{align*}
\end{corollary}

We now present some definitions that will be useful for generalizing Price's
Algorithm to consider antilayers.


Recall that an antilayer is a maximal contiguous non-empty sequence of layers of
length~$1$ and that a block is either a layer of length at least~$2$ or an
antilayer. We formalize the definition of the \emph{block sequence} of a layered
permutation~$\sigma$ as the unique sequence~$(\ell_i,\xi_i)_{i=1}^k$ of elements
of~$\naturals^*\times\{0,1\}$ such that, for every~$i\in[k]$, we have
\begin{align*}
  \left\{\sum_{j=1}^{i-1}\ell_j + t : t\in[\ell_i]\right\}\text{ is }
  \begin{dcases*}
    \text{a layer and~$\ell_i\geq 2$}, & if~$\xi_i = 0$;\\
    \text{an antilayer}, & if~$\xi_i = 1$.
  \end{dcases*}
\end{align*}
In other words, the value~$\xi_i=1$ corresponds to the ``hat'' in the notation of
block sequence defined previously.


We will also need some slightly different notions from the above. An \emph{antilayeroid} of a layered
permutation~$\sigma$ is a contiguous non-empty subsequence of an antilayer. A \emph{quasi-block} of a layered
permutation~$\sigma$ is either a layer of length at least~$2$ or an antilayeroid. A \emph{quasi-block
  sequence} or \emph{quasi-block decomposition} of a layered permutation~$\sigma$ is a
sequence~$(\ell_i,\xi_i)_{i=1}^k$ of elements of~$\naturals^*\times\{0,1\}$ such that, for every~$i\in[k]$, we
have
\begin{align*}
  \left\{\sum_{j=1}^{i-1}\ell_j + t : t\in[l_i]\right\}\text{ is }
  \begin{dcases*}
    \text{a layer and~$\ell_i\geq 2$}, & if~$\xi_i = 0$;\\
    \text{an antilayeroid}, & if~$\xi_i = 1$.
  \end{dcases*}
\end{align*}

Informally, a quasi-block decomposition of~$\sigma$ is obtained from the block
decomposition by splitting antilayers of~$\sigma$ into any number of (pairwise
disjoint) antilayeroids.

Note that, in the quasi-block decomposition, we can have two consecutive
antilayeroids (the analogous situation for block decompositions never happens).

We extend the ``hat'' notation of block decompositions to quasi-block
decompositions by using a ``hat'' to indicate which quasi-blocks are antilayeroids.

We denote the set of quasi-block decompositions of a layered permutation
by~$\cQ_\sigma$.

As an example, for~$\sigma=(321457689)$, we have
\begin{align*}
  \cQ_\sigma & = \{(3,\widehat{2},2,\widehat{2}),
  (3,\widehat{2},2,\widehat{1},\widehat{1}),
  (3,\widehat{1},\widehat{1},2,\widehat{2}),
  (3,\widehat{1},\widehat{1},2,\widehat{1},\widehat{1})\}.
\end{align*}


Note that, the block decomposition of~$\sigma$ is also a quasi-block decomposition
of~$\sigma$. Note also that quasi-block decompositions are not unique (unless every
antilayer has length~$1$).

The usefulness of this decomposition arises when we consider an occurrence of a
layered permutation~$\tau$ in another layered permutation~$\sigma$, because if two
points of~$\tau$ occur in the same block of~$\sigma$, then there must be a
quasi-block of~$\tau$ that has both points (note that this is no longer true if we
replace quasi-block by block since it is possible to split an antilayer of~$\tau$ to fit into two or more blocks of~$\sigma$).
Within this context of occurrences, we can define the concept of natural decomposition as follows.

Let~$A = \{i_1,i_2,\ldots,i_n\}_<$ be an occurrence of a layered
permutation~$\tau$ in another layered permutation~$\sigma$. The \emph{natural
  decomposition} induced by~$A$ and~$\sigma$ in~$\tau$ is the unique quasi-block
decomposition (denoted by~$\cN(A,\sigma)$) of~$\tau$ such that points of~$\tau$
are in the same quasi-block of~$\cN(A,\sigma)$ if and only if the corresponding
indices of~$A$ in the same block of~$\sigma$.

As an example, if~$\sigma = (21346587)$, $\tau = (21345)$, and~$A = \{1,2,3,4,8\}$,
then~$\cN(A,\sigma) = (2,\widehat{2},\widehat{1})$.

An easy way to obtain the natural decomposition of~$\tau$ from its occurrence~$A$
in~$\sigma$ and the block decomposition~$(\ell_i',\xi_i')_{i=1}^k$ of~$\sigma$ is to
first take~$\ell_i$ as the number of points of~$A$ in the~$i$-th block of~$\sigma$
and~$\xi_i = \xi_i'$ and then remove any zero-length quasi-blocks that arose in the
process and change any~$\xi_i$ of the layers of length~$1$ to~$1$.

In the example given, this process yields
\begin{align*}
  (2,\widehat{2},2,2) & \longrightarrow (2,\widehat{2},0,1)
  \longrightarrow (2,\widehat{2},\widehat{1}).
\end{align*}

The uniqueness of natural decompositions allows us to partition the occurrences
according to the natural decompositions that they induce.


Although ``Price's Algorithm'' always converges to the packing density, when we have
a layered permutation with an antilayer, the natural way to fit it in another layered
permutation is to put it in another antilayer. Since ``Price's Algorithm'' uses only
layers, if the extremal permutation has an antilayer, the algorithm will gradually
produce it through a sequence of small variables. With this in mind, we generalize
``Price's Algorithm'' to consider antilayers. The intuition is to construct
polynomials analogous to Price Polynomials but alternating antilayers and layers in
the ``arbitrarily large permutation'' representing the polynomial (\textit{i.e.}, variables
with odd subscripts correspond to antilayers and variables with even subscripts
correspond to layers).

\begin{definition}
  Let~$\tau\in\fS_m$ be a layered permutation of length~$m$.

  We define the \emph{Extended Price Polynomial} of order~$n\in\naturals^*$ for~$\tau$ to
  be the real polynomial over~$2n$ variables given by
  \begin{align*}
    g_{n,\tau}(x_1,x_2,\ldots,x_{2n}) & =
    m!\sum_{(\ell_j,\xi_j)_{j=1}^k\in\cQ_\tau}\sum_{\{i_1,i_2,\ldots,i_k\}_<\subset[2n]}
    \prod_{j=1}^k\frac{x_{i_j}^{\ell_j}}{\ell_j!}\One_{\{i_j\bmod 2 = \xi_j \text{ or } \ell_j = 1\}},
  \end{align*}
  where~$\One$ denotes the indicator function.

  We also define the Price Polynomial of order~$n\in\naturals^*$ for a linear
  combination of layered permutations~$f = \sum_{i=1}^ka_i\tau_i$ as
  \begin{align*}
    g_{n,f} & \equiv \sum_{i=1}^ka_ig_{n,\tau_i}.
  \end{align*}
\end{definition}

Note that the indicator function is responsible for making sure that variables with
odd subscripts correspond to antilayers (which can contain only antilayeroids
of~$\tau$) and that variables with even subscripts correspond to layers (which can
contain only layers of~$\tau$).

The Extended Price Bound defined below has yet another parameter which is responsible
for forcing some of the antilayers of the ``arbitrarily large permutation'' to have
length zero.

\begin{definition}
  Let~$f\in\reals\fS$ be a linear combination of layered permutations, let~$n\in\naturals^*$
  be a positive integer, and let~$W\subset[n]$.

  The \emph{Extended Price Bound} of order~$n$ relative to~$W$ for~$f$ is the value
  \begin{align*}
    \cL_{n,W,f} & = \max\left\{g_{n,f}(x_1,x_2,\ldots,x_{2n}) :
    \sum_{j=1}^{2n}x_j = 1\text{ and }
    \forall j\in[n], x_j\geq 0 \text{ and } \forall j\in W, x_{2j-1}=0\right\}.
  \end{align*}
\end{definition}

\begin{remark*}
  Once again this maximum exists by a compactness argument. Furthermore, note
  that~$\cL_{n,f} = \cL_{n,[n],f}$.
\end{remark*}

Let us now prove the convergence of the analogous ``Extended Price's Algorithm''.

\begin{theorem}\label{thm:extpriceconverges}
  If~$f=\sum_{t=1}^ka_t\tau_t\in\reals\fS$ is a conical combination of layered
  permutations and~$(W_n)_{n\in\naturals^*}$ is a sequence of sets such
  that~$W_n\subset[n]$ for every~$n\in\naturals^*$, then we have
  \begin{align*}
    \lim_{n\to\infty}\cL_{n,W_n,f} & = p(f).
  \end{align*}
\end{theorem}

\begin{proof}
  Note first that if~$x_1,x_2,\ldots,x_n\geq 0$ are such that~$\sum_{j=1}^nx_j=1$
  and~$q_{n,f}(x_1,x_2,\ldots,x_n) = \cL_{n,f}$, then we have
  \begin{align*}
    \cL_{n,W_n,f} & \geq g_{n,f}(0,x_1,0,x_2,\ldots,0,x_n) =
    q_{n,f}(x_1,x_2,\ldots,x_n) = \cL_{n,f}.
  \end{align*}

  Therefore, by Corollary~\ref{cor:priceconverges}, we
  have~$\liminf_{n\to\infty}\cL_{n,W_n,f}\geq p(f)$.

  Now let us prove the other inequality with~$\limsup$. Fix~$n\in\naturals^*$ and let us
  prove that~$\cL_{n,W_n,f}\leq p(f)$.

  Let~$\epsilon>0$ be an arbitrary positive real number
  and~$b_1,b_2,\ldots,b_{2n},N\in\naturals$ with~$N\neq 0$ be such that~$\sum_{j=1}^{2n}b_j=N$,
  $b_{2j-1} = 0$ for every~$j\in W_n$, $b_{2j} \neq 0$ for every~$j\in[n]$ and
  \begin{align*}
    g_{n,f}\left(\frac{b_1}{N},\frac{b_2}{N},\ldots,\frac{b_{2n}}{N}\right) & \geq
    \cL_{n,W_n,f} - \epsilon.
  \end{align*}

  Now, for every~$m\in\naturals^*$, let~$\sigma_m\in\fS_{mN}$ be the layered permutation of
  length~$mN$ and with block
  sequence~$(\widehat{mb_1},mb_2,\widehat{mb_3},mb_4,\ldots,\widehat{mb_{2n-1}},mb_{2n})$
  (if any of these numbers is zero, we remove it from the sequence to
  form~$\sigma_m$).

  Now, for every~$t\in[k]$, we count the occurrences of~$\tau_t$ in~$\sigma_m$
  according to the natural decomposition that they induce in~$\tau_t$. So we have
  \begin{align*}
    p(\tau_t,\sigma_m)
    & =
    \binom{mN}{|\tau_t|}^{-1}\sum_{(\ell_j,\xi_j)_{j=1}^k\in\cQ_{\tau_t}}\;
    \sum_{\{i_1,\ldots,i_k\}_<\subset[2n]}\;
    \prod_{j=1}^k\binom{mb_{i_j}}{\ell_j}\One_{\{i_j \bmod 2 =  \xi_j \text{ or }\ell_j = 1\}}
    \\
    & \sim \frac{|\tau_t|!}{(mN)^{|\tau_t|}}
    \sum_{(\ell_j,\xi_j)_{j=1}^k\in\cQ_{\tau_t}}\;
    \sum_{\{i_1,\ldots,i_k\}_<\subset[2n]}\;
    \prod_{j=1}^k\frac{(mb_{i_j})^{\ell_j}}{\ell_j!}
    \One_{\{i_j \bmod 2 = \xi_j \text{ or }\ell_j = 1\}}
    \\
    & =
    \left\lvert\tau_t\right\rvert!
    \sum_{(\ell_j,\xi_j)_{j=1}^k\in\cQ_{\tau_t}}\;
    \sum_{\{i_1,\ldots,i_k\}_<\subset[2n]}\;
    \prod_{j=1}^k\frac{(b_{i_j}/N)^{\ell_j}}{\ell_j!}
    \One_{\{i_j \bmod 2 = \xi_j \text{ or }\ell_j = 1\}}.
    \\
    & = g_{n,\tau_t}\left(\frac{b_1}{N},\frac{b_2}{N},\ldots,\frac{b_{2n}}{N}\right),
  \end{align*}
  where~$\sim$ means that the ratio between both sides goes to~$1$ as~$m$ goes
  to~$\infty$ and the first equality after~$\sim$ follows from the fact
  that~$\sum_{j=1}^k\ell_j = \left\lvert\tau_t\right\rvert$ for every quasi-block
  decomposition~$(\ell_j,\xi_j)_{j=1}^k$ of~$\tau_t$.

  Therefore, we have
  \begin{align*}
    p(f,\sigma_m) & = \sum_{t=1}^ka_tp(\tau_t,\sigma_m)
    \\
    & \sim \sum_{t=1}^ka_t
    g_{n,\tau_t}\left(\frac{b_1}{N},\frac{b_2}{N},\ldots,\frac{b_{2n}}{N}\right)
    \\
    & = g_{n,f}\left(\frac{b_1}{N},\frac{b_2}{N},\ldots,\frac{b_{2n}}{N}\right)
    \\
    & \geq \cL_{n,W_n,f} - \epsilon.
  \end{align*}

  So, for every~$\epsilon>0$, we have
  \begin{align*}
    p(f) & \geq \lim_{m\to\infty}p(f,\sigma_m) \geq \cL_{n,W_n,f}-\epsilon.
  \end{align*}

  Since~$\epsilon$ is arbitrary, we have~$\cL_{n,W_n,f}\leq p(f)$ for
  every~$n\in\naturals^*$, hence
  \begin{align*}
    p(f) & \geq \limsup_{n\to\infty}\cL_{n,W_n,f}.
  \end{align*}

  Therefore~$\lim_{n\to\infty}\cL_{n,W_n,f} = p(f)$.
\end{proof}


\subsection{Using generalizations}
\label{sec:usinggeneralizations}

We now introduce some notation to help in the proof of Theorem~\ref{thm:smallanti}.

\begin{notation}
  Let~$T\subset\integers$ be a set of integers and~$(x_t)_{t\in T}$ be a sequence indexed
  by elements of~$T$. Let also~$I = \{i_1,\ldots,i_k\}_<\subset T$ be a subset
  of~$T$.
  We denote by~$x_I$ the (ordered) sequence~$(x_{i_1},\ldots,x_{i_k})$.

  Furthermore, if~$x = (x_1,x_2,\ldots,x_k)$ and~$y = (y_1,y_2,\ldots,y_k)$ are
  sequences of non-negative real numbers of the same length, then we denote by~$x^y$
  the value
  \begin{align*}
    \prod_{i=1}^kx_i^{y_i},
  \end{align*}
  where~$0^0=1$.

  We abuse the notation sometimes by using a set~$T\subset\integers$ to denote the
  sequence of its elements in increasing order indexed by~$\{1,2,\ldots,|T|\}$.

  For instance, if~$T = \{1,3,6,7\}$, we have~$T_{\{3\}} = 6$.
\end{notation}

We need a straightforward technical result, which can be proved by induction in~$k$.

\begin{lemma}\label{lem:incmax}
  If~$x = (x_1,x_2,\ldots,x_k)$ and~$y = (y_1,y_2,\ldots,y_k)$ are two non-decreasing
  sequences of non-negative real numbers of the same length
  and~$z=(z_1,z_2,\ldots,z_k)$ is a permutation of the sequence~$y$, then we have
  \begin{align*}
    x^z\leq x^y.
  \end{align*}
\end{lemma}

Let us now fix some notation that will be used along the proof of
Theorem~\ref{thm:smallanti} and some auxiliary lemmas. This proof is based on proofs
by \citet[Theorem~3.3]{Hasto:PackingOtherLayered} (see Theorem~\ref{thm:hasto})
and \citet[Theorem~3.8]{Warren:OptimalPackingBehavior2Block}.

Let~$\sigma = (\widehat{a}, \ell_1, \ldots, \ell_k)$ a layered permutation with~$2
\leq a \leq \ell_1 \leq \cdots \leq \ell_k$, let~$\ell = a+\sum_{i=1}^k\ell_i$ be the
length of~$\sigma$ and let~$\ell_i = 1$ for every~$i\leq 0$.

For every~$N\in\naturals^*$, define the polynomial~$p_N\in\reals[y,x_1,x_2,\ldots,x_N]$ by
letting
\begin{align*}
  p_N(y,x_1,x_2,\ldots,x_N) & =
  \frac{\prod_{i=1}^k\ell_i!}{\ell!}g_{N,\sigma}(y,x_1,0,x_2,0,x_3,\ldots,0,x_N)
  \\
  & = \sum_{u=0}^a\frac{y^u}{u!}\sum_{\{i_1,i_2,\ldots,i_{a-u+k}\}_<\subset[N]}
  \left(\prod_{j=1}^{a-u+k}x_{i_j}^{\ell_{j-a+u}}\right)
  \\
  & = \sum_{u=0}^a\frac{y^u}{u!}
  \!\!\!\!\sum_{I\in\binom{[N]}{a-u+k}}\!\!\!\!x_I^{\ell_{[-a+u+1..k]}},
\end{align*}
where~$[-a+u+1..k]$ denotes the set~$\{-a+u+1,-a+u+2,\ldots,k\}$.

Note that
\begin{align*}
  \cL_{N,[N]\setminus\{1\},\sigma}
  = \max\Biggl\{
  & \frac{\ell!}{\prod_{i=1}^k\ell_i!}p_N(y,x_1,\ldots,x_N) :
  y + \sum_{j=1}^Nx_j = 1
  \\
  & \text{ and } \forall j\in[N],x_j\geq 0
  \text{ and } y\geq 0\Biggr\}.
\end{align*}

Furthermore, by Theorem~\ref{thm:extpriceconverges}, we have
\begin{align*}
  p(\sigma) & = \lim_{N\to\infty}\cL_{N,[N]\setminus\{1\},\sigma}.
\end{align*}

Our objective is to prove
that~$\cL_{N,[N]\setminus\{1\},\sigma}\leq\cL_{N-1,[N-1]\setminus\{1\},\sigma}$
whenever~$N>k$.

To do this, we first show that we can find an optimal point~$(y,x)$
for~$p_N$ that has some ``good'' properties.  Lemma~\ref{lem:xnondec} below shows
that we that there is an optimal~$(y,x)$ such that the coordinates of~$x$ are
increasing.  To show this, we prove that if~$x$ does not have increasing coordinnates
and we switch the position of two consecutive decreasing coordinates, then the value
of~$p_N$ does not decrease.

\begin{lemma}\label{lem:xnondec}
  For every~$N\geq k$ and~$i_0\in[N-1]$, we have
  \begin{align*}
    p_N(y,x_1,x_2,\ldots,x_N) & \leq
    p_N(y,x_1,x_2,\ldots,x_{i_0-1},x_{i_0+1},x_{i_0},x_{i_0+2},\ldots,x_N),
  \end{align*}
  whenever~$y,x_1,x_2,\ldots,x_N\geq 0$ and~$x_{i_0}\geq x_{i_0+1}$.
\end{lemma}

\begin{proof}
  Throughout this proof, we denote by~$p_N(y,x_{i_0}\leftrightarrow x_{i_0+1})$ the
  value
  \begin{align*}
    p_N(y,x_1,x_2,\ldots,x_{i_0-1},x_{i_0+1},x_{i_0},x_{i_0+2},\ldots,x_N),
  \end{align*}
  where we exchange the variables~$x_{i_0}$ and~$x_{i_0+1}$ of~$p_N$.

  We now study the difference~$p_N(y,x_{i_0}\leftrightarrow
  x_{i_0+1})-p_N(y,x_1,\ldots,x_N)$.

  Note that the summands in which~$I$ does not contain either~$i_0$ or~$i_0+1$
  cancel out.

  Furthermore, the summands of~$p_N(y,x_{i_0}\leftrightarrow x_{i_0+1})$ in
  which~$I$ contains~$i_0$ but does not contain~$i_0+1$ cancel out with the
  summands of~$p_N(y,x_1,\ldots,x_N)$ in which~$I$ contains~$i_0+1$ but does not
  contain~$i_0$ (because in the first, the values~$x_{i_0}$ and~$x_{i_0+1}$ are
  swapped).

  Analogously, the summands of~$p_N(y,x_{i_0}\leftrightarrow x_{i_0+1})$ in
  which~$I$ contains~$i_0+1$ but does not contain~$i_0$ cancel out with the
  summands of~$p_N(y,x_1,\ldots,x_N)$ in which~$I$ contains~$i_0$ but does not
  contain~$i_0+1$.

  This means that we have
  \begin{align*}
    &\hphantom{{}={}}
    p_N(y,x_{i_0}\leftrightarrow x_{i_0+1}) - p_N(y,x_1,x_2,\ldots,x_N)
    \\
    & = \sum_{u=0}^a\frac{y^u}{u!}\sum_{j=1}^{a-u+k-1}
    \!\!\!\!\!\!\!\!\!\!\!\!
    \sum_{\substack{I\in\binom{[N]}{a-u+k}:{}\\ I_{\{j\}}=i_0,\; I_{\{j+1\}}=i_0+1}}
    \!\!\!\!\!\!\!\!\!\!\!\!\!\!
    x_{I\setminus\{i_0,i_0+1\}}^{\ell_{[-a+u+1..k]\setminus\{j,j+1\}}}
   \left(x_{i_0+1}^{\ell_{j-a+u}}x_{i_0}^{\ell_{j+1-a+u}} - x_{i_0}^{\ell_{j-a+u}}x_{i_0+1}^{\ell_{j+1-a+u}}\right).
  \end{align*}

  Now, since~$x_{i_0}\geq x_{i_0+1}\geq 0$ and
  \begin{align*}
    \ell_k\geq\ell_{k-1}\geq\cdots\geq\ell_1\geq
    \ell_0\geq\ell_{-1}\geq\cdots\geq\ell_{1-a},
  \end{align*}
  we have~$x_{i_0+1}^{\ell_t}x_{i_0}^{\ell_{t+1}} \geq
  x_{i_0}^{\ell_t}x_{i_0+1}^{\ell_{t+1}}$ for every~$t<k$, hence
  \begin{align*}
    p_N(y,x_{i_0}\leftrightarrow x_{i_0+1}) - p_N(y,x_1,x_2,\ldots,x_N) & \geq 0.
    \qedhere
  \end{align*}
\end{proof}

Lemma~\ref{lem:xnondec} immediately implies that we may add the restriction
\begin{align*}
  x_1\leq x_2\leq \cdots \leq x_N
\end{align*}
to the maximum that defines~$\cL_{N,[N]\setminus\{1\},\sigma}$.

Lemma~\ref{lem:yxrel} below, shows that there is an optimal point~$(y,x)$ such that~$y\leq x_{N-k+1}$.
To do this, we use a calculus argument to show that if~$y > x_{N-k+1}$, then there
exists~$\epsilon>0$ such that if we increase~$y$
by~$\epsilon$ and decrease~$x_{N-k+1}$ by~$\epsilon$, then the value of~$p_N$ increases.

\begin{lemma}\label{lem:yxrel}
  For every~$N>k$, we may also add the restriction~$y \leq x_{N-k+1}$ to the
  maximum that defines~$\cL_{N,[N]\setminus\{1\},\sigma}$, \textit{i.e.}, we have
  \begin{align*}
    \cL_{N,[N]\setminus\{1\},\sigma} =
    \max\Biggl\{
    & \frac{\ell!}{\prod_{i=1}^k\ell_i!}p_N(y,x_1,x_2,\ldots,x_N) :
    0 \leq x_1\leq x_2\leq\cdots\leq x_N
    \text{ and }
    \\
    & 0 \leq y \leq x_{N-k+1} \text{ and } y+\sum_{j=1}^Nx_j = 1\Biggr\}.
  \end{align*}
\end{lemma}

\begin{proof}
  Suppose not and let~$y,x_1,x_2,\ldots,x_N\geq 0$ be such that~$x_1\leq
  x_2\leq\cdots\leq x_N$; $y+\sum_{j=1}^Nx_j = 1$ and
  \begin{align*}
    \cL_{N,[N]\setminus\{1\},\sigma} & =
    \frac{\ell!}{\prod_{i=1}^k\ell_i!}p_N(y,x_1,x_2,\ldots,x_N),
  \end{align*}
  and suppose that~$y>x_{N-k+1}$.

  For every~$t\in\reals$, let
  \begin{align*}
    f(t) & =
    p_N(y-t,x_1,x_2,\ldots,x_{N-k},x_{N-k+1}+t,x_{N-k+2},\ldots,x_N),
  \end{align*}
  and note that, for every~$0\leq t\leq y$, we have
  \begin{align*}
    \cL_{N,[N]\setminus\{1\},\sigma}\geq \frac{\ell!}{\prod_{i=1}^k\ell_i!}f(t),
  \end{align*}
  with equality if~$t = 0$ (we may lose the condition~$x_1\leq x_2\leq\cdots\leq
  x_N$ for~$t > 0$).

  Since~$f$ is differentiable, we must have~$f'(0)\leq 0$. But note that
  \begin{align*}
    f'(0) & =
    - \sum_{u=0}^a\frac{uy^{u-1}}{u!}
    \!\!\!\!\!\!\!\!
    \sum_{I\in\binom{[N]}{a-u+k}}
    \!\!\!\!\!\!\!\!
    x_I^{\ell_{[-a+u+1..k]}}
    + \sum_{u=0}^a\frac{y^u}{u!}
    \!\!
    \sum_{j=a-u+1}^{a-u+k}\!\!
    \sum_{\substack{I\in\binom{[N]}{a-u+k}:{}\\ I_{\{j\}}=N-k+1}}
    \!\!\!\!\!\!\!\!
    x_{I\setminus\{N-k+1\}}^{\ell_{[-a+u+1..k]\setminus\{j\}}}\ell_{j-a+u}x_{N-k+1}^{\ell_{j-a+u}-1},
  \end{align*}
  where the first sum groups terms that derived from~$y-t$ and the second sum groups
  terms derived from~$x_{N-k+1}+t$.

  We now split the first sum according to summands in which~$I$ has or not the
  element~$N-k+1$, obtaining
  \begin{align*}
    & \hphantom{{}={}}
    \sum_{u=0}^a\frac{uy^{u-1}}{u!}
    \!\!\!\!\!\!\!\!
    \sum_{I\in\binom{[N]}{a-u+k}}
    \!\!\!\!\!\!\!\!
    x_I^{\ell_{[-a+u+1..k]}}
    \\
    & =
    \sum_{u=1}^a\frac{y^{u-1}}{(u-1)!}
    \!\!\!\!
    \sum_{I\in\binom{[N]\setminus\{N-k+1\}}{a-u+k}}
    \!\!\!\!\!\!\!\!\!\!
    x_I^{\ell_{[-a+u+1..k]}}
    + \sum_{u=1}^a\frac{y^{u-1}}{(u-1)!}
    \!\!
    \sum_{j=a-u+1}^{a-u+k}
    \sum_{\substack{I\in\binom{[N]}{a-u+k}:{}\\ I_{\{j\}}=N-k+1}}
    \!\!\!\!\!\!\!\!
    x_{I\setminus\{N-k+1\}}^{\ell_{[-a+u+1..k]\setminus\{j\}}}x_{N-k+1}^{\ell_{j-a+u}}
    \\
    & =
    \sum_{u=0}^{a-1}\frac{y^{u}}{u!}
    \!\!\!\!\!\!
    \sum_{I\in\binom{[N]\setminus\{N-k+1\}}{a-u-1+k}}
    \!\!\!\!\!\!\!\!\!\!
    x_I^{\ell_{[-a+u+2..k]}}
    + \sum_{u=1}^a\frac{y^{u-1}}{(u-1)!}
    \!\!
    \sum_{j=a-u+1}^{a-u+k}
    \sum_{\substack{I\in\binom{[N]}{a-u+k}:{}\\ I_{\{j\}}=N-k+1}}
    \!\!\!\!\!\!\!\!
    x_{I\setminus\{N-k+1\}}^{\ell_{[-a+u+1..k]\setminus\{j\}}}x_{N-k+1}^{\ell_{j-a+u}},
  \end{align*}
  where in the last equality we applied the change of variables~$u\to u+1$ to the
  first sum.

  Grouping back in the original equation yields
  \begin{align*}
    f'(0) & =
    - \sum_{u=0}^{a-1}\frac{y^{u}}{u!}
    \!\!\!\!
    \sum_{I\in\binom{[N]\setminus\{N-k+1\}}{a-u-1+k}}
    \!\!\!\!\!\!\!\!
    x_I^{\ell_{[-a+u+2..k]}}
    \\ & \quad
    + \sum_{u=1}^a\frac{y^{u-1}}{(u-1)!}
    \sum_{j=a-u+1}^{a-u+k}
    \!\!
    \sum_{\substack{I\in\binom{[N]}{a-u+k}:{}\\ I_{\{j\}}=N-k+1}}
    \!\!\!\!\!\!
    x_{I\setminus\{N-k+1\}}^{\ell_{[-a+u+1..k]\setminus\{j\}}}x_{N-k+1}^{\ell_{j-a+u}-1}
    \left(\frac{y}{u}\ell_{j-a+u} - x_{N-k+1}\right)
    \\ & \quad
    + \sum_{j=a+1}^{a+k}
    \!\!\!\!
    \sum_{\substack{I\in\binom{[N]}{a+k}:{}\\ I_{\{j\}}=N-k+1}}
    \!\!\!\!\!\!
    x_{I\setminus\{N-k+1\}}^{\ell_{[-a+1..k]\setminus\{j\}}}\ell_{j-a}x_{N-k+1}^{\ell_{j-a}-1}.
  \end{align*}

  Now we give a lower bound for the last two sums. First, note that, for
  every~$u\in[a]$ and every~$a-u+1\leq j\leq a-u+k$, we have
  \begin{align*}
    \frac{y}{u}\ell_{j-a+u} - x_{N-k+1} & \geq \frac{y}{u}\ell_1 - x_{N-k+1} \geq 0,
  \end{align*}
  since~$u\leq a\leq\ell_1$ and~$y\geq x_{N-k+1}$.

  Using the first inequality and Lemma~\ref{lem:incmax}, we have
  \begin{align*}
    & \hphantom{{}={}}
    \sum_{j=a-u+1}^{a-u+k}
    \sum_{\substack{I\in\binom{[N]}{a-u+k}:{}\\ I_{\{j\}}=N-k+1}}
    \!\!\!\!\!\!
    x_{I\setminus\{N-k+1\}}^{\ell_{[-a+u+1..k]\setminus\{j\}}}x_{N-k+1}^{\ell_{j-a+u}-1}
    \left(\frac{y}{u}\ell_{j-a+u} - x_{N-k+1}\right)
    \\
    & \geq
    \sum_{j=a-u+1}^{a-u+k}
    \sum_{\substack{I\in\binom{[N]}{a-u+k}:{}\\ I_{\{j\}}=N-k+1}}
    \!\!\!\!\!\!
    x_{I\setminus\{N-k+1\}}^{\ell_{[-a+u+2..k]}}x_{N-k+1}^{\ell_{1-a+u}-1}
    \left(\frac{y}{u}\ell_1 - x_{N-k+1}\right)
    \\
    & =
    \sum_{\substack{I\in\binom{[N]}{a-u+k}:{}\\ N-k+1\in I}}
    \!\!\!\!
    x_{I\setminus\{N-k+1\}}^{\ell_{[-a+u+2..k]}}x_{N-k+1}^{\ell_{1-a+u}-1}
    \left(\frac{y}{u}\ell_1 - x_{N-k+1}\right)
    \\
    & =
    \sum_{I\in\binom{[N]\setminus\{N-k+1\}}{a-u-1+k}}
    \!\!\!\!\!\!\!\!
    x_I^{\ell_{[-a+u+2..k]}}x_{N-k+1}^{\ell_{1-a+u}-1}
    \left(\frac{y}{u}\ell_1 - x_{N-k+1}\right),
  \end{align*}
  where in the last equality we applied the change of variables~$I\to
  I\setminus\{N-k+1\}$.

  Analogously, we have
  \begin{align*}
    \sum_{j=a+1}^{a+k}
    \!\!
    \sum_{\substack{I\in\binom{[N]}{a+k}:{}\\ I_{\{j\}}=N-k+1}}
    \!\!\!\!\!\!
    x_{I\setminus\{N-k+1\}}^{\ell_{[-a+1..k]\setminus\{j\}}}\ell_{j-a}x_{N-k+1}^{\ell_{j-a}-1}
    & \geq
    \sum_{j=a+1}^{a+k}
    \!\!
    \sum_{\substack{I\in\binom{[N]}{a+k}:{}\\ I_{\{j\}}=N-k+1}}
    \!\!\!\!\!\!
    x_{I\setminus\{N-k+1\}}^{\ell_{[-a+2..k]}}\ell_1x_{N-k+1}^{\ell_{1-a}-1}
    \\
    & =
    \sum_{\substack{I\in\binom{[N]}{a+k}:{}\\ N-k+1\in I}}
    \!\!\!\!
    x_{I\setminus\{N-k+1\}}^{\ell_{[-a+2..k]}}\ell_1
    \\
    & =
    \sum_{I\in\binom{[N]\setminus\{N-k+1\}}{a-1+k}}
    \!\!\!\!\!\!\!\!
    x_I^{\ell_{[-a+2..k]}}\ell_1,
  \end{align*}
  since~$\ell_{1-a}=1$.

  Gathering all up, we have
  \begin{align*}
    f'(0) & \geq
    - \sum_{u=0}^{a-1}\frac{y^{u}}{u!}
    \!\!\!\!
    \sum_{I\in\binom{[N]\setminus\{N-k+1\}}{a-u-1+k}}
    \!\!\!\!\!\!\!\!
    x_I^{\ell_{[-a+u+2..k]}}
    \\ & \quad
    + \sum_{u=1}^a\frac{y^{u-1}}{(u-1)!}
    \!\!\!\!
    \sum_{I\in\binom{[N]\setminus\{N-k+1\}}{a-u-1+k}}
    \!\!\!\!\!\!\!\!
    x_I^{\ell_{[-a+u+2..k]}}x_{N-k+1}^{\ell_{1-a+u}-1}
    \left(\frac{y}{u}\ell_1 - x_{N-k+1}\right)
    \\ & \quad
    + \sum_{I\in\binom{[N]\setminus\{N-k+1\}}{a-1+k}}
    \!\!\!\!\!\!\!\!\!\!
    x_I^{\ell_{[-a+2..k]}}\ell_1
    \\
    & =
    \sum_{u=1}^{a-1}\frac{y^{u-1}}{(u-1)!}
    \!\!\!\!
    \sum_{I\in\binom{[N]\setminus\{N-k+1\}}{a-u-1+k}}
    \!\!
    \left(\frac{y}{u}\ell_1 - x_{N-k+1} - \frac{y}{u}\right)x_I^{\ell_{[-a+u+2..k]}}
    \\ & \quad
    +
    \frac{y^{a-1}}{(a-1)!}
    \!\!\!\!
    \sum_{I\in\binom{[N]\setminus\{N-k+1\}}{k-1}}
    \!\!
    \left(\frac{y}{a}\ell_1 - x_{N-k+1}\right)x_I^{\ell_{[2..k]}}x_{N-k+1}^{\ell_1-1}
    \\ & \quad
    + \sum_{I\in\binom{[N]\setminus\{N-k+1\}}{a-1+k}}
    \!\!\!\!\!\!\!\!\!\!
    x_I^{\ell_{[-a+2..k]}}(\ell_1-1).
  \end{align*}

  Note that, for every~$u\in[a-1]$, we have
  \begin{align*}
    \frac{y}{u}\ell_1 - x_{N-k+1} - \frac{y}{u}
    & = \frac{y}{u}(\ell_1-1) - x_{N-k+1}
    \geq \frac{\ell_1-1}{a-1}y - x_{N-k+1}
    \geq y - x_{N-k+1} > 0,
  \end{align*}
  since~$\ell_1\geq a$.

  Furthermore, note that
  \begin{align*}
    \frac{y}{a}\ell_1 - x_{N-k+1} \geq y - x_{N-k+1} > 0,
  \end{align*}
  and also~$\ell_1-1 > 0$, hence we have~$f'(0) > 0$, which is a contradiction.
\end{proof}

Lemma~\ref{lem:onlyklayers} below shows that it is enough to consider the Extended
Price Bound of order~$k$ relative to~$[k]\setminus\{1\}$. To do this, we start with
an optimal point~$(y,x)$ as provided by Lemma~\ref{lem:yxrel} and we show that if we
join the smaller layer (corresponding to~$x_1$) with the antilayer (corresponding
to~$y$) then the objective value does not decrease, \textit{i.e.}, we
have~$\cL_{N,[N]\setminus\{1\},\sigma} \leq \cL_{N-1,[N-1]\setminus\{1\},\sigma}$.

\begin{lemma}\label{lem:onlyklayers}
  For every~$N>k$, we have
  \begin{align*}
    \cL_{N,[N]\setminus\{1\},\sigma} \leq \cL_{N-1,[N-1]\setminus\{1\},\sigma}.
  \end{align*}
\end{lemma}

\begin{proof}[of Lemma~\ref{lem:onlyklayers}]
  From Lemma~\ref{lem:yxrel}, we know that there exist~$y,x_1,x_2,\ldots,x_N\geq 0$
  such that~$x_1\leq x_2\leq\cdots\leq x_N$, $y+\sum_{j=1}^Nx_j=1$, $y\leq
  x_{N-k+1}$, and
  \begin{align*}
    \cL_{N,[N]\setminus\{1\},\sigma} & =
    \frac{\ell!}{\prod_{i=1}^k\ell_i!}p_N(y,x_1,x_2,\ldots,x_N).
  \end{align*}

  We now consider what happens when we use the Extended Price Polynomial of
  order~$N-1$ on the point~$(y+x_1,x_2,x_3,\ldots,x_N)$. This corresponds to
  merging the first two blocks of~$p_N$ into an antilayer.

  Note that
  \begin{align*}
    \cL_{N-1,[N-1]\setminus\{1\},\sigma} & \geq
    \frac{\ell!}{\prod_{i=1}^k\ell_i!}p_{N-1}(y+x_1,x_2,x_3,\ldots,x_N),
  \end{align*}
  hence it is enough to prove that
  \begin{align*}
    p_{N-1}(y+x_1,x_2,x_3,\ldots,x_N) \geq p_N(y,x_1,x_2,\ldots,x_N).
  \end{align*}

  To do this, we reorganize the summands that occur in~$p_{N-1}$ and in~$p_N$ in a
  way that, when we compute their difference, each summand of~$p_{N-1}$ corresponds
  to a smaller summand of~$p_N$.

  First, we split the sum in the definition of~$p_N$ according to summands that
  have~$x_1$ and summands that do not. So note that
  \begin{align*}
    & \hphantom{{}={}}
    p_N(y,x_1,x_2,\ldots,x_N)
    \\
    & =
    \sum_{u=0}^a\frac{y^u}{u!}
    \!\!\!\!
    \sum_{I\in\binom{[N]}{a-u+k}}
    \!\!\!\!\!\!
    x_I^{\ell_{[-a+u+1..k]}}
    \\
    & =
    \sum_{u=0}^a\frac{y^u}{u!}
    \!\!\!\!
    \sum_{I\in\binom{[N]\setminus\{1\}}{a-u+k}}
    \!\!\!\!\!\!
    x_I^{\ell_{[-a+u+1..k]}}
    + \sum_{u=0}^{a-1}\frac{y^u}{u!}x_1
    \!\!\!\!\!\!\!\!
    \sum_{I\in\binom{[N]\setminus\{1\}}{a-u-1+k}}
    \!\!\!\!\!\!
    x_I^{\ell_{[-a+u+2..k]}}
    + \frac{y^a}{a!}x_1^{\ell_1}
    \!\!\!\!\!\!
    \sum_{I\in\binom{[N]\setminus\{1\}}{k-1}}
    \!\!\!\!\!\!
    x_I^{\ell_{[2..k]}}
    \\
    & =
    \sum_{u=0}^a\frac{y^u}{u!}
    \!\!\!\!
    \sum_{I\in\binom{[N]\setminus\{1\}}{a-u+k}}
    \!\!\!\!\!\!
    x_I^{\ell_{[-a+u+1..k]}}
    + \sum_{u=1}^{a}\frac{y^{u-1}}{(u-1)!}x_1
    \!\!\!\!\!\!
    \sum_{I\in\binom{[N]\setminus\{1\}}{a-u+k}}
    \!\!\!\!\!\!
    x_I^{\ell_{[-a+u+1..k]}}
    + \frac{y^a}{a!}x_1^{\ell_1}
    \!\!\!\!\!\!
    \sum_{I\in\binom{[N]\setminus\{1\}}{k-1}}
    \!\!\!\!\!\!
    x_I^{\ell_{[2..k]}}
    \\
    & =
    \sum_{u=1}^a\frac{y^{u-1}}{(u-1)!}
    \!\!\!\!
    \sum_{I\in\binom{[N]\setminus\{1\}}{a-u+k}}
    \!\!\!\!\!\!
    x_I^{\ell_{[-a+u+1..k]}}
    \left(\frac{y}{u} + x_1\right)
    + \sum_{I\in\binom{[N]\setminus\{1\}}{a+k}}
    \!\!\!\!\!\!
    x_I^{\ell_{[-a+1..k]}}
    + \frac{y^a}{a!}x_1^{\ell_1}
    \!\!\!\!\!\!
    \sum_{I\in\binom{[N]\setminus\{1\}}{k-1}}
    \!\!\!\!\!\!
    x_I^{\ell_{[2..k]}},
  \end{align*}
  where in the third equality we applied the change of variables~$u\to u-1$ to the
  second sum.

  On the other hand, we can also separate the summand of~$p_{N-1}$ that has~$u=0$,
  obtaining
  \begin{align*}
    p_{N-1}(y+x_1,x_2,x_3,\ldots,x_N) & =
    \sum_{u=1}^a\frac{(y+x_1)^u}{u!}
    \!\!\!\!\!\!
    \sum_{I\in\binom{[N]\setminus\{1\}}{a-u+k}}
    \!\!\!\!\!\!
    x_I^{\ell_{[-a+u+1..k]}}
    + \sum_{I\in\binom{[N]\setminus\{1\}}{a+k}}
    \!\!\!\!\!\!
    x_I^{\ell_{[-a+1..k]}}.
  \end{align*}

  So, computing the difference, we have
  \begin{align*}
    & \hphantom{{}={}}
    p_{N-1}(y+x_1,x_2,x_3,\ldots,x_N) - p_N(y,x_1,x_2,\ldots,x_N)
    \\
    & =
    \sum_{u=1}^a\sum_{I\in\binom{[N]\setminus\{1\}}{a-u+k}}
    \!\!\!\!\!\!
    x_I^{\ell_{[-a+u+1..k]}}
    \left(\frac{(y+x_1)^u}{u!} - \frac{y^{u-1}}{(u-1)!}\left(\frac{y}{u} + x_1\right)\right)
    - \frac{y^a}{a!}x_1^{\ell_1}
    \!\!\!\!\!\!
    \sum_{I\in\binom{[N]\setminus\{1\}}{k-1}}
    \!\!\!\!\!\!
    x_I^{\ell_{[2..k]}}.
  \end{align*}

  Now, for every~$u\in[a]$, we have
  \begin{align*}
    \left(\frac{(y+x_1)^u}{u!} - \frac{y^{u-1}}{(u-1)!}\left(\frac{y}{u} + x_1\right)\right)
    & = \frac{1}{u!}((y+x_1)^u - y^u - ux_1y^{u-1})
    = \frac{1}{u!}\sum_{v=0}^{u-2}\binom{u}{v}y^vx_1^{u-v}
    \geq 0.
  \end{align*}

  This yields
  \begin{align*}
    & \hphantom{{}={}}
    p_{N-1}(y+x_1,x_2,x_3,\ldots,x_N) - p_N(y,x_1,x_2,\ldots,x_N)
    \\
    & \geq
    \frac{1}{a!}\sum_{v=0}^{a-2}\binom{a}{v}y^vx_1^{a-v}
    \!\!\!\!\!\!
    \sum_{I\in\binom{[N]\setminus\{1\}}{k}}
    \!\!\!\!\!\!
    x_I^{\ell_{[k]}}
    - \frac{y^a}{a!}x_1^{\ell_1}
    \!\!\!\!\!\!
    \sum_{I\in\binom{[N]\setminus\{1\}}{k-1}}
    \!\!\!\!\!\!
    x_I^{\ell_{[2..k]}}.
  \end{align*}

  Note now that
  \begin{align*}
    \sum_{I\in\binom{[N]\setminus\{1\}}{k}}
    \!\!\!\!\!\!
    x_I^{\ell_{[k]}}
    & =
    \sum_{I\in\binom{[N]\setminus\{1\}}{k}}\frac{1}{k}\sum_{i\in I}x_I^{\ell_{[k]}}
    \geq
    \sum_{I\in\binom{[N]\setminus\{1\}}{k}}\frac{1}{k}
    \sum_{i\in I}x_{I\setminus\{i\}}^{\ell_{[2..k]}}x_i^{\ell_1}
    = \frac{1}{k}\sum_{i\in[N]\setminus\{1\}}
    \!\!\!\!
    x_i^{\ell_1}
    \!\!\!\!
    \sum_{\substack{I\in\binom{[N]\setminus\{1\}}{k}:{}\\ i\in I}}
    \!\!\!\!\!\!
    x_{I\setminus\{i\}}^{\ell_{[2..k]}},
  \end{align*}
  where the inequality follows from Lemma~\ref{lem:incmax}.

  Now we can apply the change of variables~$I\to I\setminus\{i\}$ and obtain
  \begin{align*}
    \sum_{I\in\binom{[N]\setminus\{1\}}{k}}
    \!\!\!\!\!\!
    x_I^{\ell_{[k]}}
    & \geq
    \frac{1}{k}
    \sum_{i\in[N]\setminus\{1\}}
    \!\!\!\!
    x_i^{\ell_1}
    \!\!\!\!
    \sum_{\substack{I\in\binom{[N]\setminus\{1\}}{k-1}:{}\\ i\notin I}}
    \!\!\!\!\!\!
    x_I^{\ell_{[2..k]}}
    \\
    & =
    \frac{1}{k}
    \!\!
    \sum_{I\in\binom{[N]\setminus\{1\}}{k-1}}
    \!\!\!\!
    x_I^{\ell_{[2..k]}}
    \!\!\!\!
    \sum_{\substack{i\in[N]\setminus\{1\}:{}\\ i\notin I}}
    \!\!\!\!\!\!
    x_i^{\ell_1}
    \\
    & \geq
    \frac{1}{k}
    \!\!
    \sum_{I\in\binom{[N]\setminus\{1\}}{k-1}}
    \!\!\!\!\!\!
    x_I^{\ell_{[2..k]}}
    x_{N-k+1}^{\ell_1},
  \end{align*}
  where the last inequality follows from the fact that the inner sum has to contain
  at least one of the summands~$x_{N-k+1}^{\ell_1},x_{N-k+2}^{\ell_1},\ldots,x_N^{\ell_1}$
  and they are all greater or equal to~$x_{N-k+1}^{\ell_1}$.

  Using this new inequality, we have
  \begin{align*}
    & \hphantom{{}={}}
    p_{N-1}(y+x_1,x_2,x_3,\ldots,x_N) - p_N(y,x_1,x_2,\ldots,x_N)
    \\
    & \geq
    \frac{1}{a!}
    \left(\frac{1}{k}\sum_{v=0}^{a-2}\binom{a}{v}y^vx_1^{a-v}x_{N-k+1}^{\ell_1}
    - y^ax_1^{\ell_1}\right)
    \!\!
    \sum_{I\in\binom{[N]\setminus\{1\}}{k-1}}
    \!\!\!\!\!\!
    x_I^{\ell_{[2..k]}}.
  \end{align*}

  Finally, since~$x_{N-k+1}\geq x_1$ and~$x_{N-k+1}\geq y$, we have
  \begin{align*}
    \frac{1}{k}\sum_{u=0}^{a-2}\binom{a}{u}y^ux_1^{a-u}x_{N-k+1}^{\ell_1} - y^ax_1^{\ell_1}
    & \geq
    \frac{1}{k}\sum_{u=0}^{a-2}\binom{a}{u}y^ax_1^{\ell_1} - y^ax_1^{\ell_1}
    \\
    & =
    \frac{2^a-a-1}{k}y^ax_1^{\ell_1} - y^ax_1^{\ell_1}
    \\
    & \geq 0,
  \end{align*}
  where the last inequality follows from the hypothesis that~$2^a-a-1\geq k$.



  Therefore we have
  \begin{align*}
    p_{N-1}(y+x_1,x_2,x_3,\ldots,x_N) & \geq p_N(y,x_1,x_2,\ldots,x_N).
    \qedhere
  \end{align*}
\end{proof}

    

Finally, Theorem~\ref{thm:smallanti} follows by a simple argument.

\begin{proof}[of Theorem~\ref{thm:smallanti}]  
  By definition, we know that~$\cL_{k,[k]\setminus\{1\},\sigma} \leq p(\sigma)$.
  On the other hand, by Lemma~\ref{lem:onlyklayers} we have 
  \begin{align*}
    \lim_{N\to\infty}\cL_{N,[N]\setminus\{1\},\sigma} &\leq \cL_{k,[k]\setminus\{1\},\sigma}.
  \end{align*}
  Then, by Theorem~\ref{thm:extpriceconverges}, we conclude that~$p(\sigma) = \cL_{k,[k]\setminus\{1\},\sigma}$.

  The value
  \begin{align*}
    p(\sigma) & =  \frac{|\sigma|!}{|\sigma|^{|\sigma|}}\,
    \frac{a^a}{a!}\prod_{i=1}^k\frac{\ell_i^{\ell_i}}{\ell_i!}.
  \end{align*}
  can now be computed with a standard calculus argument involving Lagrange
  multipliers.
\end{proof}

The corollary below follows from a simple argument presented
by \citet[Lemma~3.4]{Hasto:PackingOtherLayered} (see Corollary~\ref{cor:hasto}). In
simple words, the corollary below says that Theorem~\ref{thm:smallanti} remains valid
even if we change the order of the blocks.

\begin{corollary}\label{cor:smallantipermut}
  Let~$a,k\in\naturals^*$ be positive integers such that~$a\geq 2$ and~$2^a-a-1\geq k$.

  Let also~$\sigma$ be a layered permutation having exactly one antilayer of
  length~$a$ and~$k$ layers of lengths at least two. Furthermore, suppose the lengths
  of all these~$k$ layers are greater or equal to~$a$.

  If the antilayer is the~$j$-th block of~$\sigma$ and~$j\leq k$, then we have
  \begin{align*}
    p(\sigma) & = \cL_{k,[k]\setminus\{j\},\sigma}.
  \end{align*}
\end{corollary}

\begin{proof}
  Let~$\ell_1\leq\ell_2\leq\cdots\leq\ell_k$ be the lengths of the layers of~$\sigma$ of
  length at least~$2$ and let~$\sigma'$ be the permutation of block
  sequence~$(\widehat{a},\ell_1,\ell_2,\ldots,\ell_k)$.

  First we prove that for every~$N\in\naturals^*$, we
  have~$\cL_{N,\sigma}\leq\cL_{N,\sigma'}$ (note that we are using the common Price
  Bound).

  Let~$(r_1,r_2,\ldots,r_m)$ be the layer sequence of~$\sigma$ (note that the
  sequence includes the~$\ell_i$'s in some order and includes a sequence of~$1$'s of
  length~$a$, hence~$m=k+a$).

  Let~$x = (x_1,x_2,\ldots,x_N)$ be such that~$x_i\geq 0$ for every~$i\in[N]$,
  $\sum_{i=1}^Nx_i=1$, and
  \begin{align*}
    q_{N,\sigma}(x) & = \cL_{N,\sigma}.
  \end{align*}

  Finally, for every sequence of real numbers~$s$, let~$s_\leq$ be the ordering of~$s$
  (\textit{i.e.}, the sequence~$s_\leq$ is non-decreasing and has the same elements of~$s$).

  Note now that
  \begin{align*}
    q_{N,\sigma}(x) & \;=\;
    \frac{|\sigma|!}{\prod_{i=1}^mr_i!}
    \!\!
    \sum_{I\in\binom{[N]}{m}}
    \!\!\!\!
    x_I^{r_{[m]}}
    \;\leq\;
    \frac{|\sigma|!}{\prod_{i=1}^mr_i!}
    \!\!
    \sum_{I\in\binom{[N]}{m}}
    \!\!\!\!
    (x_I)_\leq^{(r_{[m]})_\leq}
    \\
    & \;=\;
    \frac{|\sigma|!}{\prod_{i=1}^mr_i!}
    \!\!
    \sum_{I\in\binom{[N]}{m}}
    \!\!\!\!
    (x_\leq)_I^{(r_{[m]})_\leq}
    \;=\;
    q_{N,\sigma'}(x_\leq)
    \;\leq\;
    \cL_{N,\sigma'}.
  \end{align*}

  Therefore, we have~$p(\sigma)\leq p(\sigma')$.

  \medskip

  Now, from Theorem~\ref{thm:smallanti}, we know
  that~$p(\sigma')=\cL_{k,[k]\setminus\{1\},\sigma'}$, so
  let~$y,x_1,x_2,\ldots,x_k\geq 0$ be such that~$y +\sum_{i=1}^kx_i = 1$ and
  \begin{align*}
    g_{k,\sigma'}(y,x_1,0,x_2,0,x_3,\ldots,0,x_k) & =
    \cL_{k,[k]\setminus\{1\},\sigma'}.
  \end{align*}

  Let~$\tau$ be a permutation such that the block sequence of~$\sigma$ is
  \begin{align*}
    (\ell_{\tau(1)},\ell_{\tau(2)},\ldots,\ell_{\tau(j-1)},\widehat{a},
    \ell_{\tau(j)},\ldots,\ell_{\tau(k)}).
  \end{align*}

  Note that
  \begin{align*}
    p(\sigma')
    & \;=\;
    \cL_{k,[k]\setminus\{1\},\sigma'}
    \;=\;
    g_{k,\sigma'}(y,x_1,0,x_2,0,x_3,\ldots,0,x_k)
    \\
    & \;=\;
    \frac{|\sigma'|!}{a!\prod_{i=1}^k\ell_i!}
    y^ax^{\ell_{[k]}}
    \;=\;
    \frac{|\sigma|!}{a!\prod_{i=1}^k\ell_i!}
    y^a\prod_{i=1}^kx_{\tau(i)}^{\ell_{\tau(i)}}
    \\
    & \;=\;
    g_{k,\sigma}(0,x_{\tau(1)},0,x_{\tau(2)},\ldots,0,x_{\tau(j-1)},
    y,x_{\tau(j)},0,x_{\tau(j+1)},\ldots,0,x_{\tau(k)})
    \\
    & \;\leq\;
    \cL_{k,[k]\setminus\{j\},\sigma}
    \;\leq\;
    p(\sigma).
  \end{align*}

  Therefore~$p(\sigma)=p(\sigma')=\cL_{k,[k]\setminus\{j\},\sigma}$.
\end{proof}

\begin{remark*}
  Note that we do not allow the antilayer to be the last block simply because the
  Extended Price Polynomial does not end with an antilayer.

  If the antilayer happens to be the last block of~$\sigma$, we cannot apply
  Corollary~\ref{cor:smallantipermut} directly, but it is easy to see that~$\sigma$
  must have the same packing density as the permutation whose block sequence is the
  reverse of the block sequence of~$\sigma$ (and this permutation does not have an
  antilayer as its last block).
\end{remark*}





\section{Minimization problem}
\label{sec:min}

We now study the dual problem of minimizing the density of permutations
asymptotically.

\begin{definition}
  The \emph{Minimization Price Bound} of order~$n\in\naturals^*$ for a linear combination
  of layered permutations~$f$ is the value
  \begin{align*}
    \cU_{n,f} & = \min\left\{q_{n,f}(x_1,x_2,\ldots,x_n) : \sum_{j=1}^nx_j = 1
    \text{ and } \forall j\in[n], x_j\geq 0\right\}.
  \end{align*}
\end{definition}

\begin{remark*}
  This minimum exists by a compactness argument.
\end{remark*}

For the Minimization Price Bound, we have the following analogous results.

\begin{proposition}
  \label{prop:Pricealgoprime}
  If~$f\in\reals\fS$ is a conical combination of layered permutations, then,
  for every~$n\in\naturals^*$, we have~$\cU_{n,f} \geq \cU_{n+1,f}$ and
  \begin{align*}
    \lim_{n\to\infty}\cU_{n,f} & \geq p''(f).
  \end{align*}
\end{proposition}

\begin{proofsketch}
  Analogous to the proof of Proposition~\ref{prop:Pricealgo}.
\end{proofsketch}

\begin{theorem}\label{thm:priceconvergesprime}
  If~$f\in\reals\fS$ is a conical combination of layered permutations, then we have
  \begin{align*}
    \lim_{n\to\infty}\cU_{n,f} & = p''(f).
  \end{align*}
\end{theorem}

\begin{proofsketch}
  Analogous to the proof of Corollary~\ref{cor:priceconverges}.
\end{proofsketch}

\begin{remark*}
  An analogous result is also valid for Extended Price Polynomials if we define the
  Minimization Extended Price Bound.
\end{remark*}


\begin{proof}[of Theorem~\ref{thm:minmono}]
  First note that, for every~$N\in\naturals^*$, we have
  \begin{align*}
    q_{N,\Id_l+\Rev_k}(x_1,x_2,\ldots,x_N) & =
    \ell!\!\!
    \sum_{I\in\binom{[N]}{\ell}}\prod_{i\in I}x_i
    + \sum_{i\in[N]}x_i^k.
  \end{align*}

  Let~$x_1,x_2,\ldots,x_N\geq 0$ with~$\sum_{i=1}^Nx_i=1$ be such that
  \begin{align*}
    q_{N,\Id_\ell+\Rev_k}(x_1,x_2,\ldots,x_N) & = \cU_{N,\Id_\ell+\Rev_k},
  \end{align*}
  and, without loss of generality, we may suppose that~$x_1\leq x_2\leq\cdots\leq
  x_N$ by symmetry of~$q_{N,\Id_\ell+\Rev_k}$.

  Our objective is to prove that~$\cU_{N,\Id_\ell+\Rev_k} \geq
  \cU_{N-1,\Id_\ell+\Rev_k}$ whenever~$N\geq\ell$. To do that, we group the summands
  in~$q_{N,\Id_\ell+\Rev_k}$ and~$q_{N-1,\Id_\ell+\Rev_k}$ according to which
  of~$x_1$ and/or~$x_2$ they contain. So note that
  \begin{align*}
    & \hphantom{{}={}}
    q_{N,\Id_\ell+\Rev_k}(x_1,x_2,\ldots,x_N)
    \\
    & =
    \ell!\left(\!
    x_1x_2\!\!\!\!\!\!\!\!
    \sum_{I\in\binom{[N]\setminus\{1,2\}}{\ell-2}}\prod_{i\in I}x_i
    + (x_1+x_2)\!\!\!\!\!\!\!\!
    \sum_{I\in\binom{[N]\setminus\{1,2\}}{\ell-1}}\prod_{i\in I}x_i
    +\!\!\!\!\!\!\!
    \sum_{I\in\binom{[N]\setminus\{1,2\}}{\ell}}\prod_{i\in I}x_i
    \!\right)
    + x_1^k+x_2^k +\!\!\!\!\!\!\!
    \sum_{i\in[N]\setminus\{1,2\}}
    \!\!\!\!\!\!
    x_i^k.
  \end{align*}

  On the other hand, we have
  \begin{align*}
    & \hphantom{{}={}}
    q_{N-1,\Id_\ell+\Rev_k}(x_1+x_2,x_3,x_4,\ldots,x_N)
    \\
    & =
    \ell!\left(\!
    (x_1+x_2)\!\!\!\!\!\!\!\!
    \sum_{I\in\binom{[N]\setminus\{1,2\}}{\ell-1}}\prod_{i\in I}x_i
    +\!\!\!\!\!\!\!
    \sum_{I\in\binom{[N]\setminus\{1,2\}}{\ell}}\prod_{i\in I}x_i
    \!\right)
    + (x_1+x_2)^k +\!\!\!\!\!\!\!
    \sum_{i\in[N]\setminus\{1,2\}}
    \!\!\!\!\!\!
    x_i^k.
  \end{align*}

  Subtracting the polynomials then yields
  \begin{align*}
    & \hphantom{{}={}}
    q_{N,\Id_\ell+\Rev_k}(x_1,x_2,\ldots,x_N) -
    q_{N-1,\Id_\ell+\Rev_k}(x_1+x_2,x_3,x_4,\ldots,x_N)
    \\
    & =
    \ell!x_1x_2\!\!\!\!\!\!\!\!
    \sum_{I\in\binom{[N]\setminus\{1,2\}}{\ell-2}}\prod_{i\in I}x_i
    + x_1^k + x_2^k - (x_1+x_2)^k
    \\
    & \geq
    \ell!x_1x_2\!\!\!\!\!\!\!\!
    \sum_{I\in\binom{[N]\setminus\{1,2\}}{\ell-2}}
    \!\!\!\!\!\!
    x_2^{\ell-2}
    + x_1^k + x_2^k - (x_1+x_2)^k
    \\
    & =
    \ell!x_1x_2^{\ell-1}\binom{N-2}{\ell-2}
    - \sum_{v=1}^{k-1}\binom{k}{v}x_1^vx_2^{k-v}
    \\
    & \geq
    \ell!x_1x_2^{\ell-1}\binom{N-2}{\ell-2}
    - \sum_{v=1}^{k-1}\binom{k}{v}x_1x_2^{k-1}
    \\
    & =
    x_1x_2^{\ell-1}\left(\ell!\binom{N-2}{\ell-2} - (2^k - 2)x_2^{k-\ell}\right).
  \end{align*}

  Now, since~$x_2$ is the second smallest of the~$x_i$'s, we have~$x_2\leq 2/N$,
  hence, since~$k\geq\ell$, we have
  \begin{align*}
    & \hphantom{{}={}}
    q_{N,\Id_\ell+\Rev_k}(x_1,x_2,\ldots,x_N) -
    q_{N-1,\Id_\ell+\Rev_k}(x_1+x_2,x_3,x_4,\ldots,x_N)
    \\
    & \geq
    x_1x_2^{\ell-1}\left(\ell!\binom{N-2}{\ell-2} - (2^k - 2)\left(\frac{2}{N}\right)^{k-\ell}\right).
  \end{align*}

  To prove that this value is non-negative, we consider three cases.
  \begin{proofcases}
    \case If~$N\geq\ell+1$, then~$N\geq 4$, hence
    \begin{align*}
      \ell!\binom{N-2}{\ell-2} - (2^k - 2)\left(\frac{2}{N}\right)^{k-\ell}
      & \geq
      \ell!\binom{N-2}{\ell-2} - (2^k - 2)2^{\ell-k}
      \\
      & \geq
      \ell!\frac{(N-\ell+1)^{\ell-2}}{(\ell-2)!} - 2^\ell
      \\
      & = \ell(\ell-1)(N-\ell+1)^{\ell-2} - 2^\ell,
    \end{align*}
    and since~$\ell\geq 3$ and~$N-\ell+1\geq 2$, this value is non-negative.

    \case If~$N=\ell$ and~$\ell\geq 4$, then
    \begin{align*}
      \ell!\binom{N-2}{\ell-2} - (2^k - 2)\left(\frac{2}{N}\right)^{k-\ell}
      & \geq
      \ell!\binom{N-2}{\ell-2} - (2^k - 2)2^{\ell-k}
      \\
      & =
      \ell! - 2^\ell + 2^{\ell-k+1},
    \end{align*}
    which is also non-negative.

    \case If~$N = \ell = 3$, then we have
    \begin{align*}
      & \hphantom{{}={}}
      q_{3,\Id_3+\Rev_k}(x_1,x_2,x_3) - q_{2,\Id_3,\Rev_k}(x_1+x_2,x_3)
      \\
      & =
      3!x_1x_2x_3 + x_1^k + x_2^k - (x_1 + x_2)^k
      \\
      & = 6x_1x_2x_3 - \sum_{v=1}^{k-1}\binom{k}{v}x_1^v x_2^{k-v}
      \\
      & = x_1x_2\left(6x_3 - k x_1^{k -2} - kx_2^{k-2}
      - \sum_{v=2}^{k-2}\binom{k}{v}x_1^{v-1} x_2^{k-v-1}\right)
      \\
      & = x_1x_2\left(6x_3 - k x_1^{k-2} - kx_2^{k-2}
      - \sum_{v=1}^{k-3}\binom{k}{v + 1}x_1^{v} x_2^{k-v-2}\right)
      \\
      & = x_1x_2\left(6x_3 - k\sum_{v=0}^{k-2}\binom{k-2}{v}x_1^vx_2^{k-2-v}
      + k\sum_{v=1}^{k-3}\binom{k-2}{v}x_1^vx_2^{k-2-v}
      - \sum_{v=1}^{k-3}\binom{k}{v+1}x_1^{v} x_2^{k-2-v}\right)
      \\
      & = x_1x_2\left(
      6x_3 - k(x_1 + x_2)^{k-2}
      + \sum_{v=1}^{k-3}\left(k\binom{k-2}{v} - \binom{k}{v+1}\right)x_1^{v} x_2^{k-v}
      \right),
    \end{align*}
    where in the fourth equality, we applied the change of variables~$v\to v+1$.

    Let us prove that this value is non-negative.

    Since~$x_3$ is the greatest of the~$x_i$'s, we have~$x_3\geq 1/3$,
    hence~$x_1+x_2\leq 2/3$. Since~$k(2/3)^{k-2}$ is a non-increasing function of~$k$
    when~$k\geq 2$, we have
    \begin{align*}
      6x_3 - k(x_1 + x_2)^{k-2} & \geq 2 - k\left(\frac{2}{3}\right)^{k-2}
      \geq 2 - 2\left(\frac{2}{3}\right)^{2 - 2}  =  0.
    \end{align*}

    It remains to prove that the sum~$\sum_{v=1}^{k-3}\bigl(k\binom{k-2}{v} -
    \binom{k}{v+1}\bigr)$ is non-negative. Let us prove something slightly stronger,
    namely let us prove that~$k\binom{k-2}{v}-\binom{k}{v+1}\geq 0$ for~$k\geq 2$
    and~$1\leq v\leq k-3$. So note that
    \begin{align*}
      k\binom{k-2}{v} - \binom{k}{v+1}
      & = \frac{k(k-2)!}{(v+1)!(k-1-v)!}\left((v+1)(k-v-1) - (k - 1)\right)
      \\
      & = \frac{k(k-2)!}{(v+1)!(k-1-v)!}\left(v(k - v - 2)\right)
      \\
      & \geq \frac{k(k-2)!}{(v+1)!(k-1-v)!}\left(v(k - (k - 3) - 2)\right)
      \\
      & \geq 0.
    \end{align*}
  \end{proofcases}

  Therefore, by Proposition~\ref{prop:Pricealgoprime} and
  Theorem~\ref{thm:priceconvergesprime}, we
  have~$p''(\Id_\ell+\Rev_k)=\cU_{\ell-1,\Id_\ell+\Rev_k}$.

  Now, by the definition of the Minimization Price Bound and a straightforward
  analysis argument, we have
  \begin{align*}
    \cU_{\ell-1,\Id_\ell+\Rev_k} & = \min\left\{\sum_{j=1}^{\ell-1}x_i^k : \sum_{j=1}^{\ell-1}x_j = 1
    \text{ and } \forall j\in[\ell-1], x_j\geq 0\right\}
    = \frac{1}{(\ell-1)^{k-1}}.
    \qedhere
  \end{align*}
\end{proof}

\section{Concluding remarks}
\label{sec:conc}

As we mentioned in Section~\ref{sec:usinggeneralizations},
Theorem~\ref{thm:smallanti} can be seen as a generalization of the following theorem.

\begin{theorem}[{\citet[Theorem~3.3]{Hasto:PackingOtherLayered}}]
\label{thm:hasto}
  If~$\sigma\in\fS$ is a layered permutation whose layer sequence is~$(\ell_1,\ell_2,\ldots,\ell_k)$,
  and~$\ell_1\leq\ell_2\leq\cdots\leq\ell_k$ and~$2^{\ell_1}\geq 1 + k$, then~$p(\sigma) = \cL_{k,\sigma}$.
\end{theorem}

As we also mentioned, H\"{a}st\"{o} presented the argument of the proof of
Corollary~\ref{cor:smallantipermut} to give the following corollary.

\begin{corollary}[{\citet[Lemma~3.4]{Hasto:PackingOtherLayered}}]
\label{cor:hasto}
  If~$\sigma\in\fS$ is a layered permutation with~$k$ layers, all of which have
  lengths greater or equal to~$\ell$ and we have~$2^\ell\geq 1+k$, then~$p(\sigma) =
  \cL_{k,\sigma}$.
\end{corollary}

The condition~$2^\ell\geq 1 + k$ of the above theorem can be seen as an analogous the
condition~$2^a - a - 1 \geq k$ of Theorem~\ref{thm:smallanti}. The fact that~$a+k$
is the number of layers of the permutation of block
sequence~$(\widehat{a},\ell_1,\ell_2,\ldots,\ell_k)$ leads us to the following
conjecture.

\begin{conjecture}
  Let~$\sigma\in\fS$ be a layered permutation with~$k$ layers and a total of~$b$
  blocks and suppose that every block of~$\sigma$ has length greater or equal
  to~$\ell$. If we have~$2^\ell \geq 1+k$, then there exists a
  sequence~$(\tau_n)_{n\in\naturals}$ of layered permutations of~$b$ blocks such
  that~$|\tau_n|<|\tau_{n+1}|$ for every~$n\in\naturals$ and
  \begin{align*}
    \lim_{n\to\infty}p(\sigma,\tau_n) = p(\sigma).
  \end{align*}
\end{conjecture}

Furthermore, \citet[Theorem~2.7]{AlbertAtkinsonHandleyHoltonStromquist:PackingPermutations} proved
that if~$\sigma\in\fS$ is a layered permutation with every layer of length at
least~$2$, then there is a sequence~$(\tau_n)_{n\in\naturals}$ of layered permutations
with~$|\tau_n|<|\tau_{n+1}|$ for every~$n\in\naturals$ and such
that~$\lim_{n\to\infty}p(\sigma,\tau_n) = p(\sigma)$ and the number of layers of
the~$\tau_n$'s are uniformly bounded. For clarity, we state below a simplified
version of their theorem.

\begin{theorem}[{\citet[Theorem~2.7]{AlbertAtkinsonHandleyHoltonStromquist:PackingPermutations}}]
  If~$\sigma\in\fS$ is a layered permutation with~$k$ layers and every layer of
  length at least~$2$, then there exists~$K$ such that~$p(\sigma) = \cL_{k+K,\sigma}$.
\end{theorem}

Later, \cite{Warren:OptimizingPackingThesis} defined the smallest such~$K$ to
be the \emph{packing complexity} of~$\sigma$ (denoted~$\kappa(\sigma)$).

The bound to~$\kappa(\sigma)$ provided
by \citet[Theorem~2.7]{AlbertAtkinsonHandleyHoltonStromquist:PackingPermutations},
however, seems to be quite far from the real value.

H\"{a}st\"{o}'s proof of Theorem~\ref{thm:hasto} was generalized by Warren to improve
this bound for the case when the layer sequence is non-decreasing.

\begin{theorem}[{\citet[Theorem~3.3.7]{Warren:OptimizingPackingThesis}}]
  If~$\sigma\in\fS$ is a layered permutation whose layer
  sequence is~$(\ell_1,\ell_2,\ldots,\ell_k)$
  and~$2\leq\ell_1\leq\ell_2\leq\cdots\leq\ell_k$, then~$\kappa(\sigma)\leq
  \max\{(k-1)/(2^\ell-2), 0\}$.
\end{theorem}

A natural question would then be if the proof of Theorem~\ref{thm:smallanti} can be
generalized to prove the following conjecture.
\begin{conjecture}
  If~$\sigma$ has block sequence~$(\widehat{a},\ell_1,\ell_2,\ldots,\ell_k)$
  with~$2\leq a\leq\ell_1\leq\ell_2\leq\cdots\leq\ell_k$, then there exists~$K$ such that
  \begin{align*}
    p(\sigma) & = \cL_{K,[K]\setminus\{1\},\sigma}.
  \end{align*}
\end{conjecture}

As several papers on permutation packing suggest, the family of layered permutations
is much easier to work with than the general family of permutations due to Price
Polynomials. Using the theory of permutons (\textit{i.e.}, limits of permutations,
see~\cite{HoppenKohayakawaMoreira:LimitsOfPermutationSequences}), the fact that Price
Bounds converge to~$p(\sigma)$ can be seen as a topological property of permutons as
we illustrate below.

\begin{proposition}
  If~$L_k$ is the family of permutons that can be obtained as limits of layered
  permutations with at most~$k$ layers, then~$\bigcup_{k\in\naturals}L_k$ is dense in the
  family of layered permutons (limits of layered permutations) with respect to the
  permuton topology (a sequence of permutons~$(W_n)_{n\in\naturals}$ converges to~$W$ if
  and only if~$\lim_{n\to\infty}p(\sigma,W_n)=p(\sigma,W)$ for every~$\sigma\in\fS$).
\end{proposition}

\begin{proofsketch}
  Either analogous to the proofs of Corollary~\ref{cor:priceconverges} and
  Theorem~\ref{thm:priceconvergesprime} or by a standard diagonalization argument.
\end{proofsketch}

The proposition above by itself is not interesting, what is interesting is the fact
that computing densities of layered permutations in permutons of~$L_k$ is easy
(yields Price Polynomials). In this light, Extended Price Polynomials can be seen
simply as replacing~$L_k$ with the family of permutons~$B_k$ that can be obtained as
limits of layered permutations with at most~$k$ blocks instead of~$k$ layers, which
trivially preserves the density property (since~$L_k\subset B_k$), but still
yields a family in which densities of layered permutations are easy to compute.

As mentioned before, Price Polynomials are quite a useful tool when studying
packing of layered permutations and there is still not an analogous tool for
non-layered permutations, so the following natural question arises.

\begin{question}
  Is there a family of permutons~$F$ that is dense in the family of all permutons and
  is such that computing densities of permutations in permutons of~$F$ is still easy?
\end{question}

Finally, Theorem~\ref{thm:minmono} is a double-edged knife in the problem of
minimization of monotone sequences, because, on the diagonal case (\textit{i.e.},
when~$k=\ell$) it suggests that Conjecture~\ref{conj:myers} is true
(since we have~$p''(\Id_{m+1}+\Rev_{m+1})=1/m^m$). On the other hand, on the general case, it
proves that
\begin{align*}
  p''(\Id_3 + \Rev_4) & = \frac{1}{2^3} > \frac{1}{3^2} \geq p'(\Id_3 + \Rev_4),
\end{align*}
which means that, when~$k\neq\ell$, the problem restricted to the class of layered
permutations is a distinct problem.

\bibliographystyle{abbrvnat}
\bibliography{refs}

\begin{thebibliography}{11}
\providecommand{\natexlab}[1]{#1}
\providecommand{\url}[1]{\texttt{#1}}
\expandafter\ifx\csname urlstyle\endcsname\relax
  \providecommand{\doi}[1]{doi: #1}\else
  \providecommand{\doi}{doi: \begingroup \urlstyle{rm}\Url}\fi

\bibitem[Albert et~al.(2002)Albert, Atkinson, Handley, Holton, and
  Stromquist]{AlbertAtkinsonHandleyHoltonStromquist:PackingPermutations}
M.~H. Albert, M.~D. Atkinson, C.~C. Handley, D.~A. Holton, and W.~Stromquist.
\newblock On packing densities of permutations.
\newblock \emph{Electron. J. Combin.}, 9\penalty0 (1):\penalty0 Research Paper
  5, 20, 2002.
\newblock ISSN 1077-8926.
\newblock URL
  \url{http://www.combinatorics.org/Volume_9/Abstracts/v9i1r5.html}.

\bibitem[Balogh et~al.(2013)Balogh, Hu, Lidick\`{y}, Pikhurko, Udvari, and
  Volec]{BaloghHuLidickyPikhurkoUdvariVolec:MinimumMonotone4}
J.~Balogh, P.~Hu, B.~Lidick\`{y}, O.~Pikhurko, B.~Udvari, and J.~Volec.
\newblock Minimum number of monotone subsequences of length 4 in permutations.
\newblock 2013.
\newblock URL \url{http://homepages.warwick.ac.uk/~maskat/Papers/monoSeq.pdf}.

\bibitem[Erd{\"o}s and Szekeres(1935)]{ErdosSzekeres:ProblemInGeometry}
P.~Erd{\"o}s and G.~Szekeres.
\newblock A combinatorial problem in geometry.
\newblock \emph{Compositio Math.}, 2:\penalty0 463--470, 1935.
\newblock ISSN 0010-437X.
\newblock URL \url{http://www.numdam.org/item?id=CM_1935__2__463_0}.

\bibitem[H{\"a}st{\"o}(2002/03)]{Hasto:PackingOtherLayered}
P.~A. H{\"a}st{\"o}.
\newblock The packing density of other layered permutations.
\newblock \emph{Electron. J. Combin.}, 9\penalty0 (2):\penalty0 Research paper
  1, 16, 2002/03.
\newblock ISSN 1077-8926.
\newblock URL
  \url{http://www.combinatorics.org/Volume_9/Abstracts/v9i2r1.html}.
\newblock Permutation patterns (Otago, 2003).

\bibitem[Hoppen et~al.(2013)Hoppen, Kohayakawa, Moreira, R{\'a}th, and
  Menezes~Sampaio]{HoppenKohayakawaMoreira:LimitsOfPermutationSequences}
C.~Hoppen, Y.~Kohayakawa, C.~G. Moreira, B.~R{\'a}th, and R.~Menezes~Sampaio.
\newblock Limits of permutation sequences.
\newblock \emph{J. Combin. Theory Ser. B}, 103\penalty0 (1):\penalty0 93--113,
  2013.
\newblock ISSN 0095-8956.
\newblock \doi{10.1016/j.jctb.2012.09.003}.
\newblock URL \url{http://dx.doi.org/10.1016/j.jctb.2012.09.003}.

\bibitem[Myers(2002/03)]{Myers:MinimumMonotoneSubsequences}
J.~S. Myers.
\newblock The minimum number of monotone subsequences.
\newblock \emph{Electron. J. Combin.}, 9\penalty0 (2):\penalty0 Research paper
  4, 17 pp. (electronic), 2002/03.
\newblock ISSN 1077-8926.
\newblock URL
  \url{http://www.combinatorics.org/Volume_9/Abstracts/v9i2r4.html}.
\newblock Permutation patterns (Otago, 2003).

\bibitem[Price(1997)]{Price:PackingDensitiesOfLayeredPatterns}
A.~L. Price.
\newblock \emph{Packing densities of layered patterns}.
\newblock ProQuest LLC, Ann Arbor, MI, 1997.
\newblock ISBN 978-0591-36275-6.
\newblock URL
  \url{http://gateway.proquest.com/openurl?url_ver=Z39.88-2004&rft_val_fmt=info:ofi/fmt:kev:mtx:dissertation&res_dat=xri:pqdiss&rft_dat=xri:pqdiss:9727276}.
\newblock Thesis (Ph.D.)--University of Pennsylvania.

\bibitem[Razborov(2007)]{Razborov:FlagAlgebras}
A.~A. Razborov.
\newblock Flag algebras.
\newblock \emph{J. Symbolic Logic}, 72\penalty0 (4):\penalty0 1239--1282, 2007.
\newblock ISSN 0022-4812.
\newblock \doi{10.2178/jsl/1203350785}.
\newblock URL \url{http://dx.doi.org/10.2178/jsl/1203350785}.

\bibitem[Samotij and Sudakov(2015)]{SamotijSudakov:NumberMonotoneSequences}
W.~Samotij and B.~Sudakov.
\newblock On the number of monotone sequences.
\newblock \emph{J. Combin. Theory Ser. B}, 115:\penalty0 132--163, 2015.
\newblock ISSN 0095-8956.
\newblock \doi{10.1016/j.jctb.2015.05.008}.
\newblock URL \url{http://dx.doi.org/10.1016/j.jctb.2015.05.008}.

\bibitem[Warren(2004)]{Warren:OptimalPackingBehavior2Block}
D.~Warren.
\newblock Optimal packing behavior of some 2-block patterns.
\newblock \emph{Ann. Comb.}, 8\penalty0 (3):\penalty0 355--367, 2004.
\newblock ISSN 0218-0006.
\newblock \doi{10.1007/s00026-004-0225-3}.
\newblock URL \url{http://dx.doi.org/10.1007/s00026-004-0225-3}.

\bibitem[Warren(2005)]{Warren:OptimizingPackingThesis}
D.~E. Warren.
\newblock \emph{Optimizing the packing behavior of layered permutation
  patterns}.
\newblock ProQuest LLC, Ann Arbor, MI, 2005.
\newblock ISBN 978-0542-17874-0.
\newblock URL
  \url{http://gateway.proquest.com/openurl?url_ver=Z39.88-2004&rft_val_fmt=info:ofi/fmt:kev:mtx:dissertation&res_dat=xri:pqdiss&rft_dat=xri:pqdiss:3178050}.
\newblock Thesis (Ph.D.)--University of Florida.

\end{thebibliography}
\label{sec:biblio}

\end{document}